\documentclass[11pt, reqno, english]{amsart}  
\usepackage[utf8]{inputenc}
\usepackage{amsthm}
\usepackage{amssymb}
\usepackage{amsmath}
\usepackage{graphicx}
\usepackage{tikz}
\usepackage{mathrsfs}
\usepackage{float}
\usepackage{makecell}
\usepackage{amscd}
\usepackage{multirow}
\usepackage{xspace}
\usepackage{verbatim}
\usepackage{fontenc}
\usepackage{hhline}
\usepackage{hyperref}
\usepackage{cite}
\usepackage[left=1in,right=1in,top=1in,bottom=1in]{geometry}

\usepackage[left=1in,right=1in,top=1in,bottom=1in]{geometry}

\newcommand{\mdeljoin}[1]{#1_{\overline{\Delta}}^{*2}}
\newcommand{\longJFQ}[1]{\mdeljoin{\overline{#1}}/(\Z/2\Z)}
\newcommand{\longJF}[1]{\mdeljoin{\overline{#1}}}
\DeclareMathOperator{\lk}{lk}
\DeclareMathOperator{\st}{st}

\DeclareMathOperator{\supp}{supp}

\newtheorem{theorem}{Theorem}
\newtheorem{lemma}[theorem]{Lemma}
\newtheorem{conjecture}[theorem]{Conjecture}
\newtheorem{question}[theorem]{Question}
\newtheorem{corollary}[theorem]{Corollary}
\newtheorem{definition}[theorem]{Definition}
\newtheorem{claim}[theorem]{Claim}
\newcommand{\Z}{\mathbb{Z}}
\newcommand{\N}{\mathbb{N}}
\newcommand{\R}{\mathbb{R}}
\newcommand{\Q}{\mathbb{Q}}

\renewcommand{\epsilon}{\varepsilon}

\title{Random complexes with free involution}
\author{Florian Frick}
\address[FF]{Dept.\ Math.\ Sciences, Carnegie Mellon University, Pittsburgh, PA 15213, USA}
\email{frick@cmu.edu} 

\author{Andrew Newman}
\address[AN]{Dept.\ Math.\ Sciences, Carnegie Mellon University, Pittsburgh, PA 15213, USA}
\email{anewman@andrew.cmu.edu} 

\thanks{FF was supported by NSF grant DMS 1855591, NSF CAREER Grant DMS 2042428, and a Sloan Research Fellowship.}

\date{\today}

\begin{document}

\maketitle
\begin{abstract}
\small
    We introduce a new model for random simplicial complexes which with high probability generates a complex that has a simply-connected double cover. Hence we develop a model for random simplicial complexes with fundamental group $\Z/2\Z$. We establish results about the typical asymptotic topology of these complexes. As a consequence we give bounds for the dimension~$d$ such that $\Z/2\Z$-equivariant maps from the double cover to $\R^d$ have zeros with high probability, thus establishing a random Borsuk--Ulam theorem. We apply this to derive a structural result for pairs of non-adjacent cliques in Erd\H os--R\'enyi random graphs.
\end{abstract}

\section{Introduction}
Probabilistic methods are used across numerous mathematical subfields to address questions about the typical properties of an object and to show the existence of objects with certain properties that are difficult to attain with a constructive approach. A constructive way of describing an object will necessarily introduce structure into the object, while sampling from a suitable probability space of objects will produce a representative that only has the amount of structure imposed by general Ramsey-theoretic results. Between the extremes of giving a short constructive description of an object and sampling from a space of all objects, is the problem of how to effectively sample from a space of objects that have some specific, desired structure. Here we address this question for random simplicial complexes that have precisely one non-trivial symmetry. 

The standard probability measures on simplicial complexes, such as the Linial--Meshulam--Wallach model, the random clique complex model, and the multiparameter model that interpolates between them, produce with high probability complexes without symmetry. Specifically \cite{BHK}, \cite{CostaFarberHorak}, and \cite{CFLargeComplexes2} respectively show that above a certain probability threshold each of these models produce complexes with trivial fundamental group. These complexes do not appear as quotients of symmetric complexes, while below the threshold the fundamental group behaves erratically, so that a desired symmetry for the universal cover cannot be imposed with probability bounded away from zero.

Our model is easily explained: For a (finite, simple) graph $G$, let $Z(G)$ be the simplicial complex on the same vertex set as $G$ with complete 1-skeleton (that is, include every edge) with higher-dimensional faces included according to the following rules:
\begin{itemize}
    \item A triangle $\sigma$ is included in $Z(G)$ if and only if an odd number of the edges of $\sigma$ belong to~$G$.
    \item A $d$-simplex $\sigma$ for $d \geq 2$ is included in $Z(G)$ if and only if all triangles of $\sigma$ belong to $Z(G)$.
\end{itemize}

Equivalently, the faces of $Z(G)$ correspond to subsets of the vertices of~$G$ that can be partitioned into two non-adjacent cliques. For example, if $G$ is a $5$-cycle, $Z(G)$ is the minimal triangulation of a M\"obius strip, whereas if $G$ is the complement of a $5$-cycle in the complete graph on six vertices, then $Z(G)$ is the minimal triangulation of the real projective plane~$\R P^2$. Our main results show that the twist that occurs in both of these examples is in fact a general phenomenon.

To this end we consider this construction as a new model for random simplicial complexes by taking $G$ to be an Erd\H{o}s--R\'{e}nyi random graph~$G(n,p)$, that is, a simple graph on $n$ vertices, where each possible edge appears independently with probability~$p$. 
In contrast to other random complex models, in our model we show that rather than having simple connectivity for sufficiently large~$p$, $Z(G)$ for $G \sim G(n, p)$ and $p$ large enough (but not too large) has fundamental group $\Z/2\Z$. 

We additionally study the homology groups of $Z(G)$ within this model and establish both vanishing and non-vanishing results for the different homology groups. We write $Z \sim Z(n,p)$ for the random model $Z(G)$ for $G \sim G(n,p)$. Our main result is the following:

\begin{theorem}\label{maintheorem}
For $d \geq 1$ and $1/(d + 1) < \alpha < 1/d$ with high probability $Z \sim Z(n, n^{-\alpha})$ has all of the following properties:
\begin{enumerate}
    \item $\pi_1(Z) = \Z/2\Z.$
    \item $\pi_k(Z) = 0$ for $2 \leq k \leq d$.
    \item $\pi_k(Z) \otimes \Q = 0$ for $d + 1 \leq k \leq 2d -2.$
    \item $H_{2d + 1}(Z) \neq 0$.
    \item $H_k(Z) = 0$ for all $k \geq 2d + 2.$
\end{enumerate}
\end{theorem}

Here homology groups are taken with integer coefficients. Theorem~\ref{maintheorem} fits in with an existing paradigm for random complexes in the sparse regime. For comparison we present the following result of Kahle about the random clique (or flag) complex model $X(n, p)$.
\begin{theorem}[{Kahle~\cite{KahleRandomClique, KahleRational}}]
\label{thm:kahle}
For $d \geq 3$ and $1/(d + 1) < \alpha < 1/d$ with high probability $X \sim X(n, n^{-\alpha})$ has all of the following properties:
\begin{enumerate}
    \item $\pi_1(X) = 0$.
    \item $\pi_k(X) = 0$ for $2 \leq k \leq \lfloor \frac{d}{2} \rfloor - 1$.
    \item $\pi_k(X) \otimes \Q = 0$ for $\lfloor \frac{d}{2} \rfloor \leq k \leq d - 1$.
    \item $H_{d}(X) \neq 0$.
    \item $H_{k}(X) = 0$ for all $k \geq d+1$.
\end{enumerate}
\end{theorem}
Theorem~\ref{thm:kahle}, therefore, establishes that there is some dimension $d$ depending on $\alpha$ so that all of the \emph{rational} homology lives in dimension $d$. This dimension is part of the \emph{critical dimension} phenomenon described for the multiparameter model in \cite{CFLargeComplexes3}. In the clique complex we also have that integer homology vanishes up to half of the critical dimension. Between half the critical dimension and the critical dimension itself, it is an important open problem to strengthen the rational homology vanishing statement to an integer homology vanishing statement. The stronger statement with integer homology groups vanishing is Kahle's Bouquet of Spheres conjecture.  

We prove our results by giving a model $\widetilde Z(2n,p)$ for random simplicial complexes $Z$ on $2n$ vertices that have a free involution, that is, a simplicial map $\iota\colon Z\to Z$ with $\iota\circ \iota$ is the identity and no point of the geometric realization of $Z$ is fixed by~$\iota$. The quotient by this free involution then yields a random complex sampled according to~$Z(n,p)$. Thus with high probability a complex sampled according to $\widetilde Z(2n,p)$ is the universal cover of~$Z(n,p)$ in the studied probability regime $p = n^{-\alpha}$, $0 < \alpha < 1$.

Borsuk--Ulam-type results give conditions for a simplicial complex $X$ with a free involution $\iota \colon X\to X$ such that for any continuous map $f\colon X\to\R^d$ there is a point $x\in |X|$ with $f(x) = f(\iota(x))$. The classical Borsuk--Ulam theorem states that any continuous map $f\colon S^d \to \R^d$ identifies antipodal points $x$ and $-x$ in~ the $d$-sphere~$S^d$. The Borsuk--Ulam theorem and its variants have found numerous applications across mathematics in combinatorics, discrete geometry, geometric topology, and functional analysis among others. 
With the random model $\widetilde Z(2n,p)$ for a complex with free involution, we can now ask whether a random space satisfies the Borsuk--Ulam theorem. We show:

\begin{theorem}
\label{thm:random-bu}
    For $d \geq 1$ and $1/(d + 1) < \alpha < 1/d$ with high probability a complex $\widetilde Z \sim \widetilde Z(2n, n^{-\alpha})$ with free involution $\iota\colon \widetilde Z\to \widetilde Z$ satisfies that for any continuous map $f\colon \widetilde Z \to \R^{d+1}$ there is an $x \in |\widetilde Z|$ with $f(x) = f(\iota(x))$.
\end{theorem}

This follows easily from our main result, since in this probability regime $\widetilde Z \sim \widetilde Z(2n, p)$ is at least $d$-connected with high probability. In fact, the connectivity over the rationals of such a $Z$ is about twice as large, so we conjecture the following:

\begin{conjecture}
\label{conj:bu}
    For $d \geq 1$ and $1/(d + 1) < \alpha < 1/d$ with high probability a complex $\widetilde Z \sim \widetilde Z(2n, n^{-\alpha})$ with free involution $\iota\colon \widetilde Z\to \widetilde Z$ satisfies that for any continuous map $f\colon \widetilde Z \to \R^{2d+1}$ there is an $x \in |\widetilde Z|$ with $f(x) = f(\iota(x))$.
\end{conjecture}

Given the wide applicability of the Borsuk--Ulam theorem, we hope that Theorem~\ref{thm:random-bu} will be useful when analyzing the typical behaviour in a wide array of contexts. for example. We derive the following consequence about the global structure of cliques in random graphs:

\begin{corollary}   
\label{cor:cliques}
    For $d \geq 1$, $1/(d + 1) < \alpha < 1/d$, and an Erd\H os--R\'enyi random graph $G \sim G(n, n^{-\alpha})$, with high probability for any map $f\colon V(G) \to \R^d$ from the vertices of $G$ to $\R^d$ there are two cliques $C_1, C_2$ in $G$ such that there is no edge from a vertex in $C_1$ to a vertex in $C_2$ and the convex hull of $f(C_1)$ intersects the convex hull of~$f(C_2)$.
\end{corollary}

By Conjecture~\ref{conj:bu} this should hold more generally for maps $f\colon V(G) \to \R^{2d}$. We illustrate Corollary~\ref{cor:cliques} with an example: For $\alpha < \frac13$ a random graph $G \sim G(n, n^{-\alpha})$ has a $7$-clique with high probability; Corollary~\ref{cor:cliques} in this setting asserts that when we place the vertices of $G$ in~$\R^3$, there are two non-adjacent cliques whose convex hulls intersect. 

\subsection{Organization of the paper}
We collect notation and some recurring definitions and results in probability theory, topology, and spectral (hyper-)graph theory in Section~\ref{sec:Prelims}. In Section~\ref{sec:EquivalentForm} we show that our model $Z(n,p)$ can equivalently be defined as the simplicial complex of disjoint, nonadjacent cliques in a random graph. We show that homotopy groups of $\widetilde{Z}$ vanish up to dimension~$d$ in Section~\ref{sec:IntegerHomology} and use this to derive the random Borsuk--Ulam theorem and Corollary~\ref{cor:cliques} in Section~\ref{sec:random-bu}. Section~\ref{sec:RationalHomology} applies Garland's method to derive rational homology vanishing results on $\widetilde{Z}$ in dimensions up to $2d-2$ and dimension~$2d$. That rational homology does not vanish in dimension $2d+1$ but vanishes in higher dimensions is established in Section~\ref{sec:CritialDim}. Lastly, in Section~\ref{MissingGroup} we remark that new ideas are required to determine homology in dimension $2d-1$, which is not covered by our theorems.

\section{Preliminaries}\label{sec:Prelims}
We first recall some standard notions from topology and probability. As is typically the case in the study of random spaces our results are asymptotic. To this end we use Bachmann--Landau notation. For two function $f, g : \N \rightarrow \R$, we have
\begin{itemize}
    \item $f(n) = O(g(n))$ if there is some constant $C$ so that $f(n) \leq C g(n)$ for all $n$.
    \item $f(n) = \Omega(g(n))$ if there is some constant $c$ so that $f(n) \geq c g(n)$ for all $n$.
    \item $f(n) = \Theta(g(n))$ if $f(n) = O(g(n))$ and $f(n) = \Theta(g(n))$.
    \item $f(n) = o(g(n))$ if $\lim_{n \rightarrow \infty} \frac{f(n)}{g(n)} = 0$.
    \item $f(n) = \omega(g(n))$ if $\lim_{n \rightarrow \infty} \frac{f(n)}{g(n)} = \infty$.
\end{itemize}
For a probability space parametrized by $n$, for instance $G(n, p)$, $X(n, p)$, $Z(n, p)$, we say that a property $\mathcal{P}$ holds \emph{with high probability} if the probability that a random object in our probability space has property $\mathcal{P}$ tends to 1 as $n$ goes to infinity. It will be necessary here to distinguish this notion with the stronger notion of a property holding \emph{with overwhelming probability}. To say that $\mathcal{P}$ holds with overwhelming probability means that the probability that our random object has property $\mathcal{P}$ is at least $1 - n^{-\omega(1)}$, informally speaking the probability that $\mathcal{P}$ fails to hold is super-polynomially small. 

We will prove ``with overwhelming probability" statements throughout because many of our results make use of deterministic statement regarding links of faces in random complexes, and often we will take a union bound over polynomially many links of super-polynomially unlikely bad events. 

A \emph{simplicial complex} $X$ is a set of sets that is closed under taking subsets. A set $\sigma \in X$ is called a \emph{face} of~$X$; a face with one element is a \emph{vertex}, a face with two elements an \emph{edge}. For a simplicial complex $X$ and a face $\sigma$ of $X$ the \emph{link of $\sigma$}, denoted $\lk_X(\sigma)$ or just $\lk(\sigma)$ when $X$ is clear from context is defined by
\[\lk_X(\sigma) = \{\tau \in X \mid \tau \cap \sigma = \emptyset \text{ and } \tau \cup \sigma \in X\}.\]

The specific type of deterministic statements we prove regarding links most often involve \emph{Garland's method}. Garland's method is by now a standard technique to prove homology vanishing theorems for a simplicial complex based on expansion properties of graphs of faces links in the complex. Here expansion properties of a graph are stated in terms of the eigenvalues for the normalized Laplacian of the graph.

Given a graph $G$ without isolated vertices let $A$ denote the usual adjacency matrix of $G$ and $D$ be the degree matrix of $A$. The usual graph Laplacian is the matrix $L = D - A$, and the normalized Laplacian is $\mathcal{L} = D^{-1/2}LD^{-1/2} = I - D^{-1/2}AD^{-1/2}$. It is standard that all the eigenvalues of $\mathcal{L}$ for any graph are in the interval $[0, 2]$ and the multiplicity of the eigenvalue 0 is exactly the number of connected components of $G$; for any connected component $H$ of $G$ the vector $D^{1/2}\textbf{1}_H$ is easily seen to be in the kernel of $D^{-1/2}AD^{-1/2}$. Thus the smallest eigenvalue of $\mathcal{L}$ is always zero, and the \emph{spectral gap of $G$} is the second-largest eigenvalue of $\mathcal{L}$. Garland's method is based on the following formulation of an earlier result of Garland due to Ballman--\'{S}wi\c{a}tkowski. 
\begin{theorem}[Ballmann--\'{S}wi\c{a}tkowski \cite{BallmannSwiatkowski}, Garland \cite{Garland}]
If $X$ is a simplicial complex so that every face of dimension at most $d$ is contained in a $d$-dimensional face and for every $\sigma$ a $(d - 2)$-dimensional face, $\lk(\sigma)$ is connected and its graph has spectral gap $\lambda_2$ larger than $1 - 1/d$, then $H_{d - 1}(X; \Q) = 0$.
\end{theorem}

Every simplicial complex has a geometric realization~$|X|$, where each set in $X$ is a simplex, glued according to set inclusion. We will only work with finite simplicial complexes on vertex set $[n] = \{1,2,\dots,n\}$. In this situation we can define the geometric realization $|X|$ of $X$ as a union of convex sets in $\R^n$ as follows:
\[
    |X| = \bigcup_{\sigma \in X} \mathrm{conv} \{e_i \ : \ i \in \sigma\} \subset \R^n.
\]

Here $\mathrm{conv} A$ denotes the convex hull of a set $A\subset \R^n$, and $e_1, \dots, e_n$ is the standard basis of~$\R^n$. Our notation will not distinguish between a simplicial complex $X$ and its geometric realization~$|X|$, unless this distinction is necessary. For example, we write $\pi_k(X)$ for the $k$th homotopy group of~$|X|$ with any base point. In this case, it is implicit that $X$ is path-connected, or if $X$ is a random space that it is path-connected with high probability.

Let $X$ be a simplicial complex on~$V$, $Y$ a simplicial complex on~$W$, and $f\colon V \to W$. We call $f$ \emph{simplicial map} if for every $\sigma \in X$ the image $f(\sigma)$ is a face of~$Y$. In this case we write $f\colon X\to Y$. A simplicial map $f\colon X\to Y$ induces a continuous map $|X| \to |Y|$.

For a simplicial complex $X$ on $[n]$ the \emph{join} $X^{*2}$ is the simplicial complex on $[n] \times \{-1,+1\}$ with faces $(\sigma \times \{-1\}) \cup (\tau \times \{+1\})$ for $\sigma, \tau \in X$. The \emph{deleted join} $X^{*2}_\Delta$ of $X$ is the subcomplex of $X^{*2}$ with faces $(\sigma \times \{-1\}) \cup (\tau \times \{+1\})$ for disjoint faces $\sigma, \tau \in X$.

The geometric realization of the join $X^{*2}$ is the join of the geometric realization~$|X|$, where the \emph{join} of a topological space $Y$ with itself consists of all abstract convex combinations of two ordered points in~$Y$, that is, it is the quotient of $Y \times Y \times [0,1]$ by the equivalence relation $\sim$ given by $(y,z,0) \sim (y',z,0)$ and $(y,z,1) \sim (y,z',1)$ for $y,y',z,z'\in Y$.

We will need the \emph{separated deleted join}, denoted~$\mdeljoin{X}$, which is the subcomplex of $X^{*2}$ with faces $(\sigma \times \{-1\}) \cup (\tau \times \{+1\})$ for faces $\sigma, \tau \in X$ such that $\sigma \cap \tau = \emptyset$ and no vertex of~$\sigma$ are adjacent to any vertex of~$\tau$. We refer to vertices $(v,-1)$ of $X^{*2}$ as \emph{minus vertices} and similarly to vertices $(v,+1)$ as \emph{plus vertices}. 

The join $X^{*2}$ has an involution (i.e., a $\Z/2\Z$-action) given by swapping corresponding plus and minus vertices: $\iota \colon X^{*2} \to X^{*2}$ is a simplicial map with $\iota((v,-1)) = (v,+1)$ and $\iota((v,+1)) = (v,-1)$. This involution is free when restricted to the deleted join~$X^{*2}_\Delta$, that is, for $x \in |X^{*2}_\Delta|$ we have that $\iota(x) \ne x$. The quotient of $X^{*2}_\Delta$ by this $\Z/2\Z$-action is a simplicial complex, where each vertex $v$ is identified with $\iota(v)$ and faces are identified accordingly. This quotient has the drawback that it is not homeomorphic to the $\Z/2\Z$-quotient of $|X^{*2}_\Delta|$ as a topological space; this is since there are many paths of the form $v, w, \iota(v)$, that in the topological quotient yield loops, but in the simplicial complex collapse to a single edge. For the purpose of taking quotients the separated deleted join is well-behaved: Each vertex $v$ is at distance at least three edges from $\iota(v)$, thus the geometric realization of the $\Z/2\Z$-quotient of~$\mdeljoin{X}$ is naturally homeomorphic to the $\Z/2\Z$-quotient of~$|\mdeljoin{X}|$. In this case we denote the quotient by $\mdeljoin{X}/(\Z/2\Z)$.


We will need two basic facts about homotopy groups, which we collect below.

\begin{theorem}[Hurewicz' theorem c.f. \cite{Hatcher} Theorems 4.32 and 2A.1]
\label{thm:hurewicz}
    Let $n \ge 2$, and let $X$ be an $(n-1)$-connected space. Then $\pi_n(X)$ is isomorphic to~$H_n(X, \Z)$. For a path-connected space $X$ first homology $H_1(X,\Z)$ is the abelianization of~$\pi_1(X)$.
\end{theorem}

\begin{theorem}
\label{thm:iso-hg}[c.f. Proposition 4.1 of \cite{Hatcher}]
    Let $p\colon \widetilde X \to X$ be the universal covering space of~$X$. Then for $n\ge 2$ we have that $\pi_n(\widetilde X)$ is isomorphic to $\pi_n(X)$. 
\end{theorem}

\section{Equivalent form of the random model}\label{sec:EquivalentForm}







We explain our model $\widetilde Z(2n,p)$ for a random complex with a free involution. Let $G$ be a (finite, simple) graph. We denote its \emph{flag complex} by $\overline G$, that is, $\overline G$ is the simplicial complex with ground set the vertex set of~$G$ and a face for every clique (i.e., complete subgraph) in~$G$. Now if $G \sim G(n,p)$ is an Erd\H os--R\'enyi random graph on $n$ vertices then we get a simplicial complex $\widetilde Z$ by taking the separated deleted join of the flag complex of~$G$, that is, $\widetilde Z = \mdeljoin{\overline G}$. In general, we denote by $\widetilde Z(G)$ the complex $\mdeljoin{\overline G}$ for any given graph~$G$. The complex $\widetilde Z$ has $2n$ vertices and a free $\Z/2\Z$-action that swaps each vertex with its copy. The quotient by this action is a simplicial complex; we will show that this is the model~$Z(n,p)$.


\begin{theorem}\label{DZtheorem}
For $d \geq 1$ and $1/(d + 1) < \alpha < 1/d$ with high probability $\widetilde Z \sim \widetilde Z(2n, n^{-\alpha})$ has all of the following properties:
\begin{enumerate}
    \item $\pi_k(\widetilde Z) = 0$ for $1 \leq k \leq d$.
    \item $\pi_k(\widetilde Z) \otimes \Q = 0$ for $d + 1 \leq k \leq 2d-2.$
    \item $H_{2d}(\widetilde Z; \Q) = 0.$
    \item $H_{2d + 1}(\widetilde Z) \neq 0$.
    \item $H_k(\widetilde Z) = 0$ for all $k \geq 2d + 2.$
\end{enumerate}
\end{theorem}
By Theorem~\ref{DZtheorem} for $\alpha < 1$, $\widetilde Z$ for $\widetilde Z \sim \widetilde Z(2n, n^{-\alpha})$ is simply connected with high probability and thus up to artefacts of vanishingly small probability we can think of $\widetilde Z \sim \widetilde Z(2n, n^{-\alpha})$ as the universal cover of $Z \sim Z(n, n^{-\alpha})$. So points 1, 2, and 3 of Theorem~\ref{maintheorem} follow immediately from points 1 and 2 of Theorem~\ref{DZtheorem}. Points 4 and 5 of Theorem~\ref{maintheorem} are not direct implications of Theorem~\ref{DZtheorem} but follow immediately from the proofs of points 4 and 5 of Theorem~\ref{DZtheorem}.

In the rest of this section we show that $\longJF{G}$ really is a double cover of $Z(G)$, justifying $\widetilde Z(G)$ as notation for $\longJF{G}$ since in the regime under consideration $\widetilde{Z}(G)$ is simply connected with overwhelming probability.




\begin{lemma}
\label{lem:cliques}
Let $G=(V,E)$ be a finite, simple graph. A subset $\tau \subset V$ is a face of $\mdeljoin{G}/(\Z/2\Z)$ if and only if the vertices of $\tau$ can be partitioned into two sets $\tau_1$ and $\tau_2$ so that $\tau_1 \cap \tau_2 = \emptyset$, $\tau_1$ and $\tau_2$ each induce a clique in~$G$, but no edges from $\tau_1$ to $\tau_2$ are present in~$G$.
\end{lemma}

\begin{proof}
Suppose that $\tau \in \mdeljoin{G}/(\Z/2\Z)$. Then $\tau$ is the image under the quotient of a face $(\tau_1\times\{+1\}) \cup (\tau_2\times\{-1\})$ in $\mdeljoin{\overline{G}}$. By definition of $\mdeljoin{\overline{G}}$ all edges within $\tau_1$ and all edges within $\tau_2$ are present in~$G$, but all edges between $\tau_1$ and $\tau_2$ are absent from~$G$. 

Conversely, if $\tau_1$ and $\tau_2$ are cliques of~$G$ partitioning the vertices of $\tau \subset V$, then $(\tau_1\times\{+1\}) \cup (\tau_2\times\{-1\})$ is a face of $\mdeljoin{\overline{G}}$ that maps to $\tau$ when taking the quotient. 
\end{proof}

\begin{theorem}
For any graph $G$, $\longJFQ{G} = Z(G).$ Thus $\longJF{G}$ is a double cover of $Z(G)$.
\end{theorem}
\begin{proof}
    By Lemma~\ref{lem:cliques}, $\longJFQ{G}$ has complete $1$-skeleton: for two distinct vertices $v$ and~$w$, either $(v,w)$ is an edge of $G$ and thus a $2$-clique, or it is not present as an edge in~$G$, in which case $\{v\}$ and $\{w\}$ are non-adjacent $1$-cliques.
    
    Now let $\tau \in \longJFQ{G}$ with $|\tau| \ge 3$, say $\tau = \tau_1 \cup \tau_2$, where $\tau_1$ and $\tau_2$ are (the vertices of) disjoint, nonadjacent cliques in~$G$. Let $\sigma \subset \tau$ be a subset of size three. If $\sigma \subset \tau_i$, then all pairs of vertices in $\sigma$ are connected by edges, so $\sigma$ induces three edges in~$G$. If $\sigma$ is split between $\tau_1$ and~$\tau_2$, say $|\sigma \cap \tau_1| = 1$ and $|\sigma \cap \tau_2| = 2$, then $\sigma$ induces exactly one edge in~$G$: The two vertices in $\sigma\cap \tau_2$ are connected by an edge, while no vertex in $\tau_1$ is connected to a vertex in~$\tau_2$. This implies that $\tau \in Z(G)$.
    
    Conversely, let $\tau \in Z(G)$ with $|\tau| \ge 3$. Thus any $\sigma \subset \tau$ of size three induces either one or three edges in~$G$. We need to show that the induced graph $G[\tau]$ on $\tau$ is a union of at most two cliques. Fix an edge $(v,w)$ of~$G[\tau]$. Let $\tau_1 = \{v,w\} \cup \{u \in \tau \ : \ (u,v), (u,w) \in E(G)\}$ and $\tau_2 = \tau \setminus \tau_1$. 
    
    Let $u_1, u_2 \in \tau_1$. Then $(v,u_1), (v,u_2) \in E(G)$, and so $\{v,u_1,u_2\}$ contains at least two edges of~$G$; thus also $(u_1,u_2)$ must be an edge of~$G$. Thus $\tau_1$ is clique. 
    
    Let $u \in \tau_1$ and $u' \in \tau_2$. Then one of the tuples $(u',v)$ or $(u',w)$ is not an edge in~$G$, but since $(v,w)$ is an edge and $\{u',v,w\}$ induces either one or three edges, we have that both $(u',v)$ and $(u',w)$ are not edges of~$G$. Now $(u,v)$ is an edge of~$G$. Thus by the same reasoning $(u',u)$ is not an edge, and $\tau_1$ and $\tau_2$ are nonadjacent. 
    
    Lastly, we show that also $\tau_2$ is a clique. Let $u,u' \in \tau_2$. Since $(u,v)$ and $(u',v)$ are not present as edges in~$G$, and $\{u,u',v\}$ must induce on or three edges, we have that $(u,u')$ is indeed an edge of~$G$. Thus $\tau_2$ is a clique.
\end{proof}

\section{Vanishing of the homotopy groups}\label{sec:IntegerHomology}
We show that the separated deleted join $\widetilde Z$ is (homotopically) $d$-connected, thus part~(2) of Theorem~\ref{maintheorem} follows from part~(1) of Theorem~\ref{DZtheorem}. Key to this proof is the following deterministic statement which is a symmetric analogue of~\cite[Lemma 4.2]{KahleRandomClique}.
\begin{lemma}\label{symmetricNerveLemma}
Fix $d \geq 1$. Suppose $X$ is the separated deleted join of a flag complex so that
\begin{itemize}
    \item $X$ is on at least $2d + 4$ vertices
    \item every set of $2d + 1$ vertices that do not contain an antipodal pair have a common neighbor, and
    \item every set $\{v_1, ..., v_{2d}\}$ of vertices that do not contain an antipodal pair have $\lk_X(v_1) \cap \cdots \cap \lk_X(v_{2d})$ path connected,
\end{itemize}
then $X$ is $d$-connected.
\end{lemma}

This proof makes use of the following version of the Nerve Lemma due to Bj\"orner. The \emph{nerve} of a collection of sets $U_1, \dots, U_n$ is the simplicial complex on vertex set~$[n]$ with $\sigma \subset [n]$ is a face if and only if $\bigcap_{i\in \sigma} U_i \ne \emptyset$.

\begin{theorem}[{\cite[Thm.~6]{Bjorner}}]
\label{BjornerTheorem}
Let $X$ be a connected regular $CW$ complex and $(X_i)_{i \in I}$ a family of subcomplexes that cover $X$. Suppose that for every finite nonempty intersection $X_{i_1} \cap \cdots \cap X_{i_t}$ is $(k - t + 1)$ connected for $t \geq 1$. Then $\pi_i(X)$ is isomorphic to $\pi_i(\mathcal{N})$, where $\mathcal{N}$ is the nerve of the cover, for all $i \leq k$. In particular if $\mathcal{N}$ is $k$-connected then so is $X$.
\end{theorem}
\begin{proof}[Proof of Lemma \ref{symmetricNerveLemma}]
We induct on $d$ and use Theorem \ref{BjornerTheorem}. For the base case we take $X$ on at least six vertices so that the common neighbors of any pair of non-antipodal vertices induce a connected graph and every set of three vertices which do not contain an antipodal pair have a common neighbor. We cover $X$ by (closed) vertex stars and apply Theorem \ref{BjornerTheorem}. Stars are contractible, so we can start by checking that 2-fold intersection of vertex stars are path-connected. If $u$ and $v$ are adjacent in $X$ then by the flag condition $\st(u) \cap \st(v) = \st(\{u, v\})$ and the star of a simplex is contractible. If $u$ and $v$ are not adjacent then $\st(u) \cap \st(v) = \lk(u) \cap \lk(v)$. If $u$ and $v$ are antipodal then $\lk(u) \cap \lk(v) = \emptyset$ and there is nothing to check as antipodal vertices in a separated deleted join have no common neighbors. Thus $u$ and $v$ are non-antipodal, non-adjacent vertices so $\lk(u) \cap \lk(v)$ is path connected by assumption. 

Any set of three vertices $\{u, v, w\}$ that have no antipodal pair have a common neighbor so $\st(u) \cap \st(v) \cap \st(w) \neq \emptyset$ in this case and so by Theorem \ref{BjornerTheorem}, it suffices to check that the nerve of this cover of $X$ is simply-connected. 

By the assumption on $X$ the nerve is an induced flag subcomplex of the (ordinary) deleted join of a simplex with $X$ having at least six vertices. This will be simply connected by the following claim, which we will also use in the inductive step.
\begin{claim}\label{SimplexDelJoinClaim}
If $X$ is an induced (flag) subcomplex of the deleted join of a simplex and $V(X) \geq 2d + 4$ then $X$ is $d$-connected.
\end{claim}
\begin{proof}
 If $X$ has at least $d + 2$ vertices on, without loss of generality, the plus side, then we cover $X$ by plus vertex stars. Since the plus side of $X$ forms a simplex the nerve of this cover is a simplex on at least $d + 2$ vertices, so it is $d$-connected. And by the flag condition $\st(v_1) \cap ... \cap \st(v_t) = \st(\{v_1 \cap \cdots \cap v_t\})$ is contractible for any $\{v_1, ..., v_t\}$ which are plus vertices since all simplices on the plus side are included and $X$ is flag.
\end{proof}

We now turn our attention to the inductive step. As before we apply Theorem \ref{BjornerTheorem} to the covering of $X$ by \emph{all} vertex stars. Let $v_1, ..., v_t$ be a collection of $k$ vertices so that $\st(v_1) \cap \cdots \cap \st(v_t)$ is nonempty. Then $v_1, ..., v_t$ contains no antipodal pair. If $v_1, ..., v_t$ form a clique in $X$ then $\st(v_1) \cap \cdots \cap \st(v_t) = \st(\{v_1, ..., v_t\})$ by the flag condition. In this case then $\st(v_1) \cap \cdots \cap \st(v_t)$ is contractible. Thus we assume that $v_1, ..., v_t$ do not induce a clique and therefore $\st(v_1) \cap \cdots \cap \st(v_t) = \lk(v_1) \cap \cdots \cap \lk(v_t)$. Within $\lk(v_1) \cap \cdots \cap \lk(v_t)$ we have that every set of $2d - t + 1$ vertices not containing an antipodal pair have a common neighbor and that every set of $2d - t$ vertices induce a connected graph among their common neighbors. However, $\lk(v_1) \cap \cdots \cap \lk(v_k)$ contains no antipodal pairs since antipodal pairs in the separated deleted join do not have any common neighbors. So $\lk(v_1) \cap \cdots \cap \lk(v_t)$ is just a clique complex in which every set of $2d - t + 1$ vertices have a common neighbor. Thus a result of Meshulam \cite{MeshulamNerve}, cited as Theorem 3.1 in \cite{KahleRandomClique}, $\lk(v_1) \cap \cdots \cap \lk(v_t)$ is $(d - \lceil \frac{t + 1}{2} \rceil)$-connected. We only need to check that it is $(d - t + 1)$-connected and this holds as long as $t - 1 \geq \lceil \frac{t+1}{2} \rceil$, i.e. as long as $t \geq 3$. For $t = 2$, we have that $\lk(v_1) \cap \lk(v_2)$ is a flag complex in which every $2(d - 1) + 1$ vertices have a common neighbor and every the common neighbors of every set of $2(d - 1)$ vertices induce a connected graph, so by \cite[Lemma 4.2]{KahleRandomClique}, $\lk(v_1) \cap \lk(v_2)$ is $(d - 1)$-connected which is what we need to apply Theorem \ref{BjornerTheorem}. All that remains to check is that the nerve given by covering by vertex stars is $d$-connected. The nerve is an induced subcomplex of the deleted join of a simplex, but by the condition that $X$ is on at least $2d + 4$ vertices we have that the nerve of $X$ is $d$-connected by Claim \ref{SimplexDelJoinClaim}. 
\end{proof}

We are now ready to prove part (1) of Theorem \ref{DZtheorem}. As we will apply Lemma \ref{symmetricNerveLemma}, we will ultimately want to understand the structure of graphs induced on common neighbors of vertices in the separated deleted join of a random flag complex. If $X$ is a flag complex and $\{+v_1, ..., +v_k, -u_1, ..., -u_{\ell}\}$ is a set of vertices in $\mdeljoin{X}$ then the graph induced on the common neighbors of $\{+v_1, ..., +v_k, -u_1, ..., -u_{\ell}\}$ is the following auxiliary graph obtained from the graph of $X$: Let $A$ be the set of common neighbors in $X$ of $v_1, ..., v_k$ other than $u_1,...,u_{\ell}$ and neighbors of (any of) $u_1, ..., u_{\ell}$ and $B$ be the common neighbors in $X$ of $u_1, ..., u_{\ell}$ other than $v_1, ..., v_k$ and the neighbors of $v_1, .., v_k$. The graph induced on the common neighbors of $\{+v_1, ..., +v_k, -u_1, ..., -u_{\ell}\}$ is the graph on vertex set $A \cup B$ with an edge between two vertices of $A$ if and only if that particular edge is present in $X$ and likewise for edges between vertices of $B$, but each possible edge between $A$ and $B$ is included if and only if that edge is \emph{excluded} from $X$.

In the case that $X$ is the flag complex of an Erd\H{o}s--R\'{e}nyi random graph then for a fixed set $\{+v_1, ..., +v_k, -u_1, ..., -u_{\ell}\}$ of vertices in $\mdeljoin{X}$, the graph induced on $\lk(+v_1) \cap \cdots \cap \lk(+v_k) \cap \lk(-u_1) \cap \dots \cap \lk(-u_{\ell})$ is a random graph sampled by the following model:

\begin{definition}
For $n \in \mathbb{N}$, $p_A, p_B, p_{eA}, p_{eB}, p_{eAB}$ probabilities, let $H(n, p_A, p_B, p_{eA}, p_{eB}, p_{eAB})$ be the probability space on graphs on at most $n$ vertices sampled by starting with ground set $[n]$ and first building two disjoint vertex sets $A$ and $B$ by including each vertex of $[n]$ in $A$ with probability~$p_A$, in $B$ with probability $p_B$, or excluded from the graph entirely with probability $1 - p_A - p_B$. Once $A$ and $B$ are set, each edge between vertices of $A$ is included independently with probability $p_{eA}$, each edge between vertices of $B$ is included independently with probability $p_{eB}$, and edge with one endpoint in $A$ and one endpoint in $B$ is included with probability $p_{eAB}.$

As it will be a common special case, for $p_A$, $p_B$ and $q$ we will use $H(n, p_A, p_B, q)$ to refer to $H(n, p_A, p_B, q, q, 1 - q)$.
\end{definition}

The proof of part (1) of Theorem \ref{DZtheorem} is almost immediate from Lemma \ref{symmetricNerveLemma} and the following:
\begin{lemma}\label{AuxGraphLemma1}
Fix $d \geq 1$, if $\alpha < 1/d$ then with overwhelming probability if $k + \ell = 2d + 1$, $H \sim H(n - (2d + 1), n^{-k\alpha}(1 - n^{-\alpha})^{\ell}, n^{-\ell \alpha}(1 - n^{-\alpha})^k, n^{-\alpha})$ is nonempty and if $k + \ell = 2d$ then $H \sim H(n - 2d, n^{-k\alpha}(1 - n^{-\alpha})^{\ell}, n^{-\ell \alpha}(1 - n^{-\alpha})^k, n^{-\alpha})$ is connected.
\end{lemma}
\begin{proof}
First suppose that $k + \ell = 2d + 1$ then without loss of generality $k \leq d$. In this case then $|A|$ in $H$ is binomially distributed with $n - (2d + 1)$ trials and success probability $\Theta(n^{-k\alpha})$ with $k\alpha < 1$. So the probability that $A$ is empty is $(1 - \Theta(n^{-k\alpha}))^{n - (2d + 1)} = \exp(\Omega(-n^{1 - k\alpha})$. Thus with overwhelming probability $H$ is nonempty. 

Now we suppose that $k + \ell = 2d$. If $k = \ell = d$ then we have that with overwhelming probability $A$ and $B$ both have size at least $\Theta(n^{1 - d\alpha}) \geq n^{\epsilon} \rightarrow \infty$. The edges going between $A$ and $B$ form a random bipartite graph with each edge included with probability $1 - n^{-\alpha} \rightarrow 1$. It is straightforward to check that such a random bipartite graph (where the edge probability goes to 1 and the size of the sets on each side go to infinity at the same rate) is connected. This fact also follows from Lemma~\ref{SymmetricCase}. This handles the case that $k = \ell$.  On the other hand suppose that without loss of generality that $k \leq d - 1.$ Then the graph induced on $A$ is a two parameter random graph with ground set $[n]$ each vertex included with probability $n^{-k\alpha}(1 - n^{-\alpha})^{\ell}$ and each edge included with probability $n^{-\alpha}$, as $n^{-k  \alpha}(1 - n^{-\alpha})^{\ell}n^{-\alpha} = \omega\left(\frac{\log n}{n}\right)$ we have that with overwhelming probability the graph induced on $A$ is connected. This follows from the Erd\H{o}s--R\'{e}nyi connectivity threshold, and is covered by Lemma \ref{twoparametergeneral}. Conditioned on $A$ inducing a connected graph on approximately its expected number of vertices, both of which are overwhelmingly likely, we observe that each vertex of $B$, if there are any, sends a binomially distributed number of edges to $A$ with $\Theta(n^{1 - k \alpha})$ trials and success probability $1 - n^{-\alpha} \rightarrow 1$. By ordinary large deviation inequalities then with overwhelming probability every vertex of $B$ sends an edge to the connected graph $A$ and hence $H$ is connected.
\end{proof}

\begin{proof}[Proof of part (1) of Theorem \ref{DZtheorem} and parts (1) and (2) of Theorem \ref{maintheorem}]
We combine Lemma \ref{symmetricNerveLemma} and Lemma~\ref{AuxGraphLemma1}. Suppose $\alpha < 1/d$ and let $X \sim X(n, n^{-\alpha})$ we show that $\widetilde Z = \mdeljoin{X}$ satisfies the assumptions of Lemma~\ref{symmetricNerveLemma}. If $S$ is a set of vertices in $\widetilde Z$ of size $2d + 1$ without an antipodal pair then the graph induced on the common neighbors of $S$ are distributed as 
$$H \sim H(n - (2d + 1), n^{-k\alpha}(1 - n^{-\alpha})^{\ell}, n^{-\ell \alpha}(1 - n^{-\alpha})^k, n^{-\alpha})$$ with $k$ the number of plus vertices and $\ell$ the number of minus vertices. Thus by Lemma~\ref{AuxGraphLemma1} with overwhelming probability the vertices of $S$ have a common neighbor. As this happens with overwhelming probability for each of the polynomially many choices of~$S$, each such $S$ has a common neighbor. Verifying the connectivity condition when $|S| = 2d$ follows from Lemma \ref{AuxGraphLemma1} in the same way. Therefore by Lemma \ref{symmetricNerveLemma}, $\widetilde Z$ is $d$-connected.

As $\widetilde Z$ is simply-connected and double covers~$Z$, we have that $\pi_1(Z) = \Z/2\Z$, so we have part~(1) of Theorem~\ref{maintheorem}. Part~(2) of Theorem~\ref{maintheorem} follows from the fact that maps from simply-connected spaces into $Z$ factor through $\widetilde Z$.
\end{proof}


\section{A random Borsuk--Ulam theorem}
\label{sec:random-bu}

Here we prove Theorem~\ref{thm:random-bu}, that is, that for $d \geq 1$ and $1/(d + 1) < \alpha < 1/d$ a complex $\widetilde Z \sim \widetilde Z(2n, n^{-\alpha})$ satisfies the Borsuk--Ulam theorem for maps to $\R^{d+1}$ with high probability. We have just established part~(1) of Theorem~\ref{DZtheorem} that $\widetilde Z$ is $d$-connected with high probability. If $X$ is a $d$-connected space with a free involution $\iota\colon X \to X$ then there is a map $h\colon S^{d+1} \to X$ such that $h(-x) = \iota(h(x))$ for all $x \in S^{d+1}$. Thus with high probability we have a map $h \colon S^{d+1} \to \widetilde Z$ with $h(-x) = \iota(h(x))$ for all~$x \in S^{d+1}$. Given any map $f\colon \widetilde Z \to \R^{d+1}$, the composition $f\circ h\colon S^{d+1} \to \R^{d+1}$ identifies two antipodal points, $f(h(-x)) = f(h(x))$ for some~$x$, by the Borsuk--Ulam theorem. Then $f(\iota(y)) = f(y)$ for $y =h(x)$.

To derive Corollary~\ref{cor:cliques} as a consequence, we use the standard proof scheme for topological Radon--type results; see Matou\v sek~\cite{Matousek} for an introduction. let $G \sim G(n, n^{-\alpha})$ be an Erd\H os--R\'enyi random graph, with $\alpha$ as above. Let $f\colon V(G) \to \R^d$ be a map. By linearly extending this map onto faces we obtain a map $F\colon |\overline G| \to \R^d$. This induces a map $h \colon |\longJF{G}| \to \R^{d+1}$ defined by $h(\lambda x + (1-\lambda) y) = (\lambda, \lambda F(x))$. Here we think of any point $z$ in the join $|\overline G^{*2}|$ as an abstract convex combination $\lambda x + (1-\lambda) y$ with $\lambda \in [0,1]$ and $x,y \in |\overline G|$.

Since $G \sim G(n, n^{-\alpha})$ we have that $\longJF{G} \sim \widetilde Z(2n, n^{-\alpha})$ and thus by Theorem~\ref{thm:random-bu} there is an $\lambda x + (1-\lambda) y \in |\longJF{G}|$ with $h(\lambda x + (1-\lambda) y) = h(\iota(\lambda x + (1-\lambda) y))$. This implies $\lambda = 1-\lambda$ and thus $\lambda = \frac12$, and that $\frac12 F(x) = \frac12 F(y)$ so $F(x) = F(y)$. The points $x$ and $y$ are in disjoint, nonadjacent faces $\sigma$ and $\tau$ of the clique complex $\overline{G}$. Thus $F(\sigma)$ and $F(\tau)$, which are the convex hulls of $f(\sigma)$ and $f(\tau)$, intersect in~$\R^d$.

\section{Vanishing of the rational homology groups}\label{sec:RationalHomology}
The goal of this section is to prove the following theorem which implies points (2) and (3) of Theorem \ref{DZtheorem} and therefore (3) of Theorem \ref{maintheorem}. 
\begin{theorem}\label{RationalHomologyVanishingTheorem}
For $d \geq 1$ if $\alpha < 1/d$ then with overwhelming probability $H_{2d}(\widetilde Z, \Q) = 0$ for $\widetilde Z \sim \widetilde Z(2n, n^{-\alpha})$ and if $\alpha < 1/(d+ 1)$ then with overwhelming probability $H_{2d - 1}(\widetilde Z; \Q) = 0$. 
\end{theorem}
The key tool to prove Theorem \ref{RationalHomologyVanishingTheorem} is Garland's method. The first paper to apply Garland's method to random complexes was a paper of Hoffman, Kahle, and Paquette \cite{HKP}. The main result of \cite{HKP} is a result about the concentration of the eigenvalues of an Erd\H{o}s--R\'{e}nyi random graph. While there is extensive literature on eigenvalues of random graphs, the key contribution of \cite{HKP} is a concentration of measure result for the eigenvalue of an Erd\H{o}s--R\'{e}nyi random graph. Specifically the authors of \cite{HKP} show that the nontrivial eigenvalues of an Erd\H{o}s--R\'{e}nyi random graph are concentrated around 1 with overwhelming probability. The following formulation is weaker than the full statement of the main theorem of \cite{HKP}, but is simpler to state and will be strong enough to prove the results that we need. The two parameter Erd\H{o}s--R\'enyi random graph $G(n; p_0, p_1)$ is sampled by starting with a ground set of possible vertices $[n]$, including each vertex independently  with probability $p_0$ and including each edge between existing vertices with probability $p_1$.
\begin{lemma}[Lemma 22 of \cite{DochtermannNewman}, immediate corollary of main result of \cite{HKP}]\label{twoparametergeneral}
Let $G \sim G(n; p_0, p_1)$ be a two parameter Erd\H{o}s--R\'{e}nyi random graph with the property that $p_0 n \rightarrow \infty$ as $n \rightarrow \infty$. Furthermore assume that $p_0$ and $p_1$ satisfy
\[p_0p_1 = \omega \left( \frac{\log n}{n} \right).\]
Then for any $M \geq 0$ and $\epsilon > 0$ we have the following:
\begin{itemize}
\item For $n$ large enough $\lambda_2(G) \geq 1 - \epsilon$ with probability at least $1 - (np_0)^{-M}$.
\item In particular $G$ is connected with high probability. 
\end{itemize}
\end{lemma}
A useful corollary to this is the following. This corollary is a ``two parameter" version of one of the main result of \cite{KahleRational}.

\begin{corollary}\label{TwoParameterRationalConnectivity}
Let $G \sim G(n; p_0, p_1)$ be a two parameter Erd\H{o}s--R\'{e}nyi random graph with the property that $p_0 n \rightarrow \infty$ as $n \rightarrow \infty$. Furthermore assume that 
$p_0$ and $p_1$ satisfy
\[p_0p_1 = \omega \left( \left( \frac{\log n}{n} \right)^{1/d} \right)\]
Then with overwhelming probability the flag complex of $G$ is $(d - 1)$-rationally homologically connected.
\end{corollary}

\begin{proof}
Fix $k \leq d - 1$. We apply Garland's method and Lemma \ref{twoparametergeneral} to show that $H_k(\overline{G}; \Q) = 0$ with overwhelming probability. If this holds for every $k \leq d - 1$ with overwhelming probability then it holds for \emph{all} such $k$ simultaneously with overwhelming probability. 

We must check the necessary spectral condition and that any face of dimension at most $d$ is contained in a face of dimension $d$ both hold with overwhelming probability. The pure-dimensionality condition follows from well-known results about random graphs. We give a proof here though for the sake of completeness. We show that for $G \sim G(n; p_0, p_1)$ as given, with overwhelming probability $G$ does not contain a maximal $\ell$-clique for $\ell$ smaller than $d + 1$. For a fixed set of $\ell$ vertices the probability that $G$ contains a maximal $\ell$ clique on those vertices is 
\[p_0^{\ell}p_1^{\binom{\ell}{2}}(1 - p_0p_1^k)^{n - \ell} \leq \exp(- p_0p_1^{\ell} n) \leq \exp(-\omega(\log n)). \]
By union bound of the polynomially many choices for a possible $\ell$-clique we have the pure-dimensionality condition. 

Now we consider the link of a $(k - 1)$-face in $\overline{G}$. We observe that this is a flag complex $\overline{H}$ for $H \sim G(n - k; (p_0p_1)^k, p_1)$. By Lemma \ref{twoparametergeneral}, it suffices to verify that \[p_0^k(p_1)^{k + 1} = \omega \left(\frac{\log n}{n} \right).\]
This follows by 
\[p_0^{k}(p_1)^{k + 1} \geq p_0^{k + 1}p_1^{k + 1} = \omega\left(\left( \frac{ \log n}{n} \right)^{(k + 1)/d} \right)\]
with $k + 1 \leq d.$
\end{proof}

For a graph $G$ the construction of $\widetilde Z(G) = \mdeljoin{\overline G}$ has that for any face $\sigma * \tau = (\sigma\times \{-1\}) \cup (\tau \times \{+1\})$, the link $\lk(\sigma * \tau)$ is the flag complex induced on the common neighbors in $\widetilde Z(G)$ of the vertices of $\sigma*\tau$, since $\widetilde Z(G)$ is a flag complex. We have already discussed the connection between common neighbors of sets of vertices in $\widetilde Z(G)$ for $G \sim G(n, p)$ and the random graph model introduced as $H(n, p_A, p_B, q)$. This auxiliary model will again be relevant here.

In order to prove Theorem \ref{RationalHomologyVanishingTheorem} using Garland's method alone we would have to show, for the vanishing homology statement about $H_{2d}$, that for every set of $2d$ vertices in $\widetilde Z$ without an antipodal pair, the common link is a good spectral expander. For a choice of $2d$ vertices with $k$ of them positive and $\ell$ of them negative, the common neighborhood in $\widetilde Z(n, n^{-\alpha})$ is distributed as 
\[H(n - 2d, n^{-k\alpha}(1 - n^{-\alpha})^{\ell}, n^{-\ell \alpha}(1 - n^{-\alpha})^k, n^{-\alpha}).\]
The structure of this graph can vary greatly depending on $k$ and $\ell$ even keeping the requirement that $k + \ell = 2d$. To discuss all these possible cases, we introduce the following notation:

\begin{definition}
For a graph $G$ and its separated deleted join $\widetilde Z = \widetilde Z(G)$ we denote by $\Pi_{k, \ell}(DZ)$, or just $\Pi_{k, \ell}$ if the complex is clear from context, the subcomplex of $\widetilde Z$ generated by all faces $\sigma*\tau$ where the cardinality of $\sigma$ is $k$ and the cardinality of $\tau$ is $\ell$.
\end{definition}

The $2d$th homology group of $\widetilde Z$ depends only on the $(2d + 1)$st skeleton of $\widetilde Z$, so it suffices to show that $H_{2d}(\bigcup_{k + \ell = 2d + 2} \Pi_{k, \ell}(\widetilde Z)) = 0$. We can similarly split up the $2d$th skeleton in order to study the $(2d - 1)$st homology group. The key to the argument will first be two reduce to studying only $\Pi_{d + 1, d + 1}$ (for $2d$ homology) and $\Pi_{d, d + 1} \cup \Pi_{d + 1, d}$ (for $(2d - 1)$-homology). This reduction is the following lemma.

\begin{lemma}\label{GluingStep}
For $d \geq 1$ if $\alpha < 1/d$ then with overwhelming probability every $2d$-cycle in $\widetilde Z$, for $\widetilde Z \sim \widetilde Z(2n, n^{-\alpha})$, is rationally homologous to a cycle in $\Pi_{d + 1, d + 1}$ and if $\alpha < 1/(d + 1)$ then with overwhelming probability every $(2d - 1)$-cycle is rationally homologous to a cycle in $\Pi_{d + 1, d} \cup \Pi_{d, d + 1}$. 
\end{lemma}

\begin{proof}[Proof of Lemma \ref{GluingStep}]
Let $C$ be a $2d$-cycle in $\widetilde Z$ with $\tau*\sigma$ a face with $|\tau| > d + 1$ and so that there is no other face $\tau' * \sigma'$ of $C$ with $\tau \subseteq \tau'$ on the plus side of $\widetilde Z$. (If no such choice for $\tau*\sigma$ exists then try again with $\tau' * \sigma'$ until we get a $\tau$ that is maximal on the plus side.) For this choice of $\tau*\sigma$ let $\Sigma_{\tau} := \{\sigma \mid \tau*\sigma \in C\}$ on the minus side of $\widetilde Z$. We claim that $\Sigma_{\tau}$ forms a $(2d - |\tau|)$-cycle on the minus side of $\widetilde Z$. This can be seen from the maximality of $\tau$ giving a block structure to the boundary matrix $\partial_{2d}(\widetilde Z)$. By the maximality of $\tau$ the rows and columns of $\partial_{2d}(\widetilde Z)$ may be rearranged so that 
\[\partial_{2d}(\widetilde Z) = \begin{pmatrix} M_1 & 0 \\ M_3 & M_4 \end{pmatrix}\]
where the columns are indexed first by $2d$-faces with $\tau$ as their plus vertices and the rows indexed first by $(2d - 1)$-faces with $\tau$ as their plus vertices. By this block structure then the restriction of $C$ to $\tau*\Sigma_{\tau}$ is in the kernel of $M_1$. But also by the block structure $M_1$ is just the $(2d - |\tau|)$-boundary matrix of the flag complex of the underlying graph with the vertices of $\tau$ and the neighbors of the vertices of $\tau$ deleted. This is the flag complex of a two-parameter random graph $G(n - |\tau|, (1 - n^{-\alpha})^{|\tau|}, n^{-\alpha})$. As $\alpha < 1/d$, we have 
\[(1 - n^{-\alpha})^{|\tau|}n^{-\alpha} = \omega \left(\left(\frac{\log n}{n}\right)^{1/d} \right).\]
So with overwhelming probability the flag complex of the underlying graph with $\tau$ and the neighbors of $\tau$ deleted is rationally $(d - 1)$-connected by Corollary~\ref{TwoParameterRationalConnectivity}. Thus there is a rational $(2d - |\tau| + 1)$-chain $\overline{\Sigma_{\tau}}$ in $X$ so that $\partial(\overline{\Sigma_{\tau}}) = \Sigma_{\tau}$.

Now $\partial (\tau*\overline{\Sigma_{\tau}}) = \tau*\Sigma_{\tau} + \partial \tau * \overline{\Sigma_{\tau}}$. So $C$ is homologous to $C - \tau*\Sigma_{\tau} - \partial \tau * \overline{\Sigma_{\tau}}$ replacing $\tau*\Sigma_{\tau}$ in the support with $\partial \tau * \overline{\Sigma_{\tau}}$. Thus we have replaced faces with $\tau$ as their positive face with faces with facets of $\tau$ as their positive faces. So we can do this until we are left with faces $\tau * \sigma$ so that $|\tau| \leq d$ and $||\tau| - |\sigma|| \leq 1$. At that point we are left with $C$ belonging to $\Pi_{d + 1, d + 1}$ by the fact that the underlying flag complex has every face of dimension at most $d$ belonging to a face of dimension~$d$. (See the proof of Corollary \ref{TwoParameterRationalConnectivity} for the proof of this standard fact about maximal cliques in the Erd\H{o}s--R\'enyi random graph.) The argument for $(2d - 1)$-cycles works in the same way. 
\end{proof}

Now we apply Garland's method to show that $\Pi_{d + 1, d + 1}$ has no rational homology in degree $2d$ and that $\Pi_{d, d + 1} \cup \Pi_{d + 1, d}$ has no rational homology in degree $2d - 1$. For $\sigma*\tau$ of dimension $2d - 1$ in $\Pi_{d + 1, d + 1}$, we have two cases to consider, either $|\sigma| = |\tau| = d$ or (without loss of generality), $|\sigma| = d + 1$ and $|\tau| = d - 1$. In the first case the link within $\Pi_{d + 1, d + 1}$ is a particular random bipartite graph, and in the second case the link is an Erd\H{o}s--R\'enyi random graph. The second case is essentially handled by results of \cite{HKP}, but the first case requires establishing new bounds on the spectral gap for certain random bipartite graphs. Within the expander graph literature there are results on eigenvalues of random bipartite graphs. Such results in quite a general setting are found in work of Chung and Radcliffe \cite{ChungRadcliffe}. However, nothing that we found had the \emph{with overwhelming probability} type of estimates that we needed. The way that we go about establishing the spectral expansion property that we need with the right concentration of measure inequality uses the following result of Bilu and Linial \cite{BiluLinial}. 

\begin{theorem}[Lemma 3.3 of \cite{BiluLinial}]
Let $A$ be an $n \times n$ symmetric matrix so that the $\ell_1$-norm of each row in $A$ is at most $K$, and all diagonal entries of $A$ are, in absolute value, $O(\epsilon(\log(K/\epsilon) + 1))$. Assume that for any two vectors $u, v \in \{0, 1\}^n$ with $\supp(u) \cap \supp(v) = \emptyset:$
\[\frac{|u^tAv|}{||u||||v||} \leq \epsilon.\]
Then the spectral radius of $A$ is 
$O(\epsilon( \log(K/\epsilon) + 1)).$
\end{theorem}
This theorem will also allow us to establish results about the spectral expansion for certain random bipartite graphs from the vertex expansion. The main component to make this precise is the following theorem. As a reminder, the normalized Laplacian of a graph has an eigenvalue of 2 for exactly each bipartite component, so when considering bipartite connected graphs there are two trivial eigenvalues.
\begin{theorem}\label{bipartiteexpansion}
If $G$ is a connected bipartite graph with bipartition $X$, $Y$ with $\overline{d}$ as the average degree of vertices in $X$ and $\overline{c}$ as the average degree of vertices in $Y$. And for any $U \subseteq X$, $V \subseteq Y$, 
\[\left|e(U, V) - \frac{|U||V| \overline{c}}{|X|}\right| \leq \epsilon \sqrt{|U||V|} \]
Then there is some constant $C$ so that the spectral gap of the normalized Laplacian of $G$ is at least
\[1 - \frac{C\epsilon \log((\Delta(G) + \overline{d} + \overline{c})/\epsilon) + C}{\delta(G)} - \frac{2\sigma_X\sigma_Y}{\delta(X)\delta(Y)},\]
where $\sigma_X$ and $\sigma_Y$ are the standard deviations of the degree sequence of $X$ and $Y$ respectively, and $\Delta$, $\delta$ denote the maximum and minimum vertex degrees. 
\end{theorem}
\begin{proof}
Let $A$ be the adjacency matrix of $G$ and $D$ be the degree matrix of $G$. We show that all but two eigenvalues of $D^{-1/2}AD^{-1/2}$ are small. As $G$ is a connected bipartite graph, $D^{-1/2}AD^{-1/2} = I - \mathcal{L}$ has the two trivial eigenvalues $1$ and $-1$. The eigenvector of $1$ is $D^{1/2}\textbf{1}$ where $\textbf{1}$ denotes the all 1's vector. The eigenvector of $-1$ is $D^{1/2} \hat{\textbf{1}}$ where $\hat{\textbf{1}}(v) = 1$ for $v \in X$ and $\hat{\textbf{1}}(v) = -1$ for $v \in Y$. Since $D^{-1/2}AD^{-1/2}$ is a symmetric matrix any nontrivial eigenvector $g$ of $D^{-1/2}AD^{-1/2}$ satisfies
\[\langle g, D^{1/2}\textbf{1}_X \rangle = \langle g, D^{1/2}\textbf{1}_Y \rangle = 0,\]
where $\textbf{1}_X$ and $\textbf{1}_Y$ are the indicator vectors for $X$ and $Y$. So we suppose $\langle g, D^{1/2}\textbf{1}_X \rangle = \langle g, D^{1/2}\textbf{1}_Y \rangle = 0$ and bound
\[\frac{|\langle AD^{-1/2} g, D^{-1/2}g \rangle|}{||g||^2} .\]
Let $\overline{A}$ be the \emph{expectation matrix} for $G$ (in using this we take inspiration from \cite{ChungRadcliffe}). That is for $x \in X$ and $y \in Y$, $\overline{A}_{xy} = \overline{A}_{yx} = \frac{\overline{c}}{|X|} = \frac{\overline{d}}{|Y|}$ and all other entries of $\overline{A}$ are zero. By the triangle inequality 
\[\frac{|\langle AD^{-1/2} g, D^{-1/2}g \rangle|}{||g||^2} \leq \frac{|\langle (A - \overline{A})D^{-1/2} g, D^{-1/2}g \rangle|}{||g||^2} + \frac{|\langle \overline{A}D^{-1/2} g, D^{-1/2}g \rangle|}{||g||^2}.\]

For the first summand we use Lemma 3.3 of \cite{BiluLinial}. The matrix $A - \overline{A}$ is symmetric, has that the $\ell_1$-norm of each column is at most $\Delta(G) + \overline{d} + \overline{c}$, and has all diagonal entries equal to zero. Now for any indicator vectors $u$ and $v$ for disjoint sets $U$ and $V$, 
\begin{eqnarray*}
\frac{|u^t(A - \overline{A})v|}{||u|| ||v||} &=& \frac{|e(U, V) - (|V \cap X||U \cap Y| + |V \cap Y||U \cap X|)(\overline{c}/|X|)|}{\sqrt{|U||V|}} \\
&=& \frac{|e(U \cap X, V \cap Y) + e(U \cap Y, V \cap X) - (|V \cap X||U \cap Y| + |V \cap Y||U \cap X|)(\overline{c}/|X|)|}{\sqrt{|U||V|}} \\ 
&\leq& \frac{|e(U \cap X, V \cap Y) - |U \cap X||V \cap Y|(\overline{c}/|X|)|}{\sqrt{|U \cap X||V \cap Y|}}\frac{\sqrt{|U \cap X||V \cap Y|}}{\sqrt{|U||V|}} \\
&& + \frac{|e(U \cap Y, V \cap X) - |U \cap Y||V \cap X| (\overline{c}/|X|)|}{\sqrt{|U \cap Y||V \cap X|}} \frac{\sqrt{|U \cap Y||V \cap X|}}{\sqrt{|U||V|}} \\
&\leq& \epsilon \frac{\sqrt{|U \cap X||V \cap Y|}}{\sqrt{|U||V|}} + \epsilon  \frac{\sqrt{|U \cap Y||V \cap X|}}{\sqrt{|U||V|}} \\
&\leq& 2\epsilon.
\end{eqnarray*}
Thus by Lemma 3.3 of \cite{BiluLinial} the spectral radius of $A - \overline{A}$ is at most $O(2\epsilon \log ((\Delta(G) + \overline{d} + \overline{c})/\epsilon) + 1)$. Therefore
\begin{eqnarray*}
\frac{|\langle (A - \overline{A})D^{-1/2} g, D^{-1/2} g \rangle|}{||g||^2} &\leq& \frac{|\langle (A - \overline{A})D^{-1/2} g, D^{-1/2} g \rangle|}{\delta(G)||D^{-1/2}g||^2} \\
&\leq& O\left(\frac{1}{\delta(G)}(2\epsilon \log ((\Delta(G) + \overline{d} + \overline{c})/\epsilon) + 1)\right).
\end{eqnarray*}
Next we turn our attention to bounding 
\[\frac{|\langle \overline{A}D^{-1/2}g , D^{-1/2} g \rangle|}{||g||^2}.\]
We note that by construction $\overline{A}$ is a rank 2 matrix, and thus has $n - 2$ linearly independent eigenvectors for the eigenvalue zero. Moreover it is easy to check that the two nontrivial eigenvectors of $\overline{A}$ are $f: [n] \rightarrow \R$ so that $f(v) = \sqrt{|Y|}$ if $v \in X$ and $f(v) = \sqrt{|X|}$ if $v \in Y$, for the eigenvalue $\sqrt{\overline{c} \overline{d}}$ and $\hat{f}$ defined by $\hat{f}(v) = f(v)$ for $v \in X$ and $\hat{f}(v) = -f(v)$ for $v \in Y$. Thus any vector orthogonal to both $\textbf{1}_X$ and $\textbf{1}_Y$ is in the kernel of $\overline{A}$. Thus we may write 
$$D^{-1/2} g = a \textbf{1}_X + b \textbf{1}_Y + \textbf{k}$$ 
where $\textbf{k}$ is a vector in the kernel of $\overline{A}$ orthogonal to both $\textbf{1}_X$ and $\textbf{1}_Y$.
Thus 
\begin{eqnarray*}
|\langle \overline{A}D^{-1/2} g, D^{-1/2}g \rangle| &=& |\langle a \overline{A} \textbf{1}_X + b \overline{A} \textbf{1}_Y, a\textbf{1}_X + b \textbf{1}_Y, + \textbf{k} \rangle|  \\
&=& |\langle a \overline{c} \textbf{1}_Y + b \overline{d} \textbf{1}_X, a \textbf{1}_X + b \textbf{1}_Y + \textbf{k} \rangle |\\
&=& |ab\overline{c}|Y| + ab \overline{d}|X|| = 2|ab||E(G)|  
\end{eqnarray*}
Here we have for $a$ and $b$: 
\[a = \langle g, D^{-1/2} \textbf{1}_X \rangle/|X|\]
and
\[b = \langle g, D^{-1/2} \textbf{1}_Y \rangle/|Y|. \]
So,
\begin{eqnarray*}
|a| &=& \frac{|\sum_{v \in X} g(v) (-\overline{d}/\sqrt{\deg(v)})|}{\overline{d}|X|} \\
&=& \frac{|\sum_{v \in X} g(v) ((-\deg(v) - \overline{d} + \deg(v))/\sqrt{\deg(v)}) |}{\overline{d}|X|} \\
&=& \frac{|-\sum_{v \in X} g(v)\sqrt{\deg(v)} + \sum_{v \in X} g(v) (\deg(v) - \overline{d})/\sqrt{\deg(v)}|}{\overline{d}|X|}.
\end{eqnarray*}
Recall that $\langle g, D^{1/2} \textbf{1}_X \rangle = 0$, and so we obtain
\begin{eqnarray*}
|a| &\leq& \frac{||g|| \sqrt{\sum_{v \in X} (\deg(v) - \overline{d})^2}}{\overline{d}|X| \sqrt{\delta(G)}}\\
&=& \frac{||g|| \sqrt{|X|}\sigma_X}{\overline{d}|X|\sqrt{\delta(X)}}.
\end{eqnarray*}
Similarly,
\[|b| \leq \frac{||g|| \sigma_Y}{\overline{c} \sqrt{|Y|} \sqrt{\delta(Y)}}.\]
So putting everything together,
\begin{eqnarray*}
\frac{|\langle\overline{A}D^{-1/2}g, D^{-1/2}g \rangle|}{||g||^2} &\leq& 2 \frac{\sigma_X \sigma_Y}{\overline{d}\overline{c} \sqrt{|X||Y|} \sqrt{\delta(X)\delta(Y)}} E(G) \\
&=& \frac{|X| \overline{d} + |Y| \overline{c}}{\overline{d}\overline{c} \sqrt{|X||Y|}} \frac{\sigma_X\sigma_Y}{\sqrt{\delta(X)\delta(Y)}} \\
&=& \frac{1}{\sqrt{dc}}\left(\left(\frac{\overline{d}|X|}{\overline{c}|Y|}\right)^{1/2} + \left(\frac{\overline{c}|Y|}{\overline{d}|X|}\right)^{1/2}\right) \frac{\sigma_X\sigma_Y}{\sqrt{\delta(X)\delta(Y)}} \\
&\leq& \frac{2\sigma_X\sigma_Y}{\delta(X)\delta(Y)}.
\end{eqnarray*}
\end{proof}
Turning our attention back to the overall strategy we observe that within $\Pi_{d + 1, d + 1}(\widetilde Z)$ $\lk(\tau*\sigma)$ with $|\tau| = |\sigma| = d$ and $\widetilde Z \sim \widetilde Z(2n, p)$ is the auxiliary random graph $$H \sim H(n - 2d, n^{-d \alpha}(1 - n^{-\alpha})^d, n^{-d \alpha}(1 - n^{-\alpha})^d, 0, 0, 1 - n^{-\alpha}).$$ We apply Theorem \ref{bipartiteexpansion} to a random graph in this model.

\begin{lemma}\label{SymmetricCase}
For $d \geq 1$ and $\alpha < 1/d$ and $\delta > 0$ with overwhelming probability $$H \sim H(n - 2d, n^{-d \alpha}(1 - n^{-\alpha})^d, n^{-d \alpha}(1 - n^{-\alpha})^d, 0, 0, 1 - n^{-\alpha})$$ has spectral gap of the normalized Laplacian larger than $1 - \delta$. 
\end{lemma}
\begin{proof}
The randomness of the size of $A$ and $B$ can be effectively ignored because each of $A$ and $B$ are binomially distributed in size with expectation $\Theta(n^{1 - d \alpha})$ which goes to infinity for $\alpha < 1/d$, so by large deviation estimates of a binomial random variable we may assume that $1.01 n^{1 - d \alpha} \geq |A|, |B| \geq 0.99 n^{1 - d \alpha}$. Moreover for convenience of the rest of the argument we assume by symmetry that $|A| \geq |B|$. Conditioned on the high probability event that $|A|$ and $|B|$ are both bounded as above we can just consider $A$ and $B$ to be set and we examine the random bipartite graph on $(A, B)$ with each crossing edge included independently with probability $1 - n^{-\alpha}$.

We apply Theorem \ref{bipartiteexpansion} to this random graph. Let $U \subseteq A$ and $V \subseteq B$ be fixed and we bound the probability that
\[\left| e(U, V) - \frac{|U||V|\overline{c}}{|A|} \right| \leq \sqrt{C|A|p|U||V|}\]
for $C$ a fixed constant and $\overline{c} = \frac{1}{|B|}\sum_{b \in B} \deg(b)$.

By triangle inequality,
\begin{eqnarray*}
\left|e(U, V) - \frac{|U||V| \sum_{b \in B} \deg(v)}{|A||B|}\right| &\leq& |e(U, V) - p|U||V|| + \left|p|U||V| - \frac{|U||V| \sum_{b \in B} \deg(v)}{|A||B|}\right| \\
&=& |e(U, V) - p|U||V|| + \frac{|U||V|}{|A||B|} \left||A||B|p - \sum_{b \in B} \deg(b) \right|.
\end{eqnarray*}

Now $e(U, V)$ is distributed as a binomial with $|U||V|$ trials and success probability $p$, so the probability that
\[|e(U, V) - p|U||V|| \geq \sqrt{C |A| p |U||V|}/2\]
is at most $\exp(-C_1 |A|)$ for some constant $C_1$ that grows larger with $C$. 

Indeed by Chernoff bound (c.f. Theorem 23.6 of \cite{FriezeRandomGraphs}), 
\[\Pr(e(U, V) \leq p|U||V| - \sqrt{C |A| p|U||V|}/2) \leq \exp\left(-\frac{3 C|A|p|U||V|}{16 p|U||V|}\right) = \exp(-3C|A|/16)\]
and for 
\[\sqrt{Cp|A||U||V|} \leq 3p|U||V|\] we have the upper tail
\[\Pr(e(U, V) \geq p|U||V| + \sqrt{Cp|A||U||V|}) \leq \exp \left(-\frac{C|A||U||V| p}{4(p|U||V| + p|U||V|/2)}\right) \leq \exp(-C|A|/5).\]

But if $\sqrt{Cp|A||U||V|} \geq 2p|U||V|$ then
\[\Pr(e(U, V) \geq p|U||V| + \sqrt{Cp|A||U||V|})/2 \leq \Pr(e(U, V) \geq (p + p) |U||V|).\]
And as $p$ is tending to 1, for $p$ larger than 1/2 this probability is 0 since $e(U, V)$ is always at most~$|U||V|$.

Setting $C$ large enough $C_1$ becomes large enough that we can take a union bound over all $2^{|A| + |B|}$ choices of $|U|$ and $|V|$ to have that the above inequality holds for all $U$ and $V$.

For the second summand we have that 

\[\left||A||B|p - \sum_{b \in B} \deg(b) \right| = \left| e(A, B) - p|A||B| \right|\]
which is at most $\sqrt{C|A|^2p|B|}/2$ by what we showed in the first part in the special case that $A = U$ and $B = V$. Thus
\[\frac{|U||V|}{|A||B|} \left||A||B|p - \sum_{b \in B} \deg(b) \right| \leq \frac{\sqrt{C|A|p|U||V|}}{2}\]
as $\frac{|U||V|}{|A||B|} \leq \frac{\sqrt{|U||V|}}{\sqrt{|A||B|}}$.

Thus with overwhelming probability our random graph satisfies the vertex expansion condition in the assumptions of Theorem \ref{bipartiteexpansion} with $\epsilon = \sqrt{C |A| p}$ for $C$ a large constant. Trivially the maximum degree of our graph is at most $|A|$ and the minimum degree is at least $0.99|B|$ with overwhelming probability for $p = 1 - n^{-\alpha}$ sufficiently close to 1. Thus with overwhelming probability the spectral gap of our random bipartite graph is at least

\[1 - \frac{\sqrt{C |A|} \log(2(2|A| + |B|)/\sqrt{C|A|})}{0.99|B|} - \frac{2\sigma_A\sigma_B}{0.99|A|B|}\]
as $|A| \approx |B| = \Theta(n^{1 - d\alpha})$ we have that the middle term goes to zero with $n$. For the term involving the standard deviations of the degree sequence, it is intuitively clear that it out to go to zero by the fact that $\sigma_A^2$ has expectation $|A|p(1 - p)$ with $(1 - p) \rightarrow 0$. Thus by Markov's inequality, with \emph{high} probability this last term is at most $\delta$ for any fixed constant $\delta$. However we have to prove that it is small with overwhelming probability, which takes a small amount of work. We fix $\delta > 0$ and bound the probability that 
\[\frac{\sigma_A}{|B|} \leq \delta.\]

We have for $n$ large enough
\begin{eqnarray*}
\sigma_A &=& \sqrt{\frac{1}{|A| - 1}\sum_{a \in A}(\deg(a) - \overline{d})^2} \\
&\leq& 1.01 \frac{1}{\sqrt{|A|}} \sqrt{\sum_{a \in A} (\deg(a) - \overline{d})^2} \\
&\leq& \frac{1.01}{\sqrt{|A|}} \left( \sqrt{\sum_{a \in A} (\deg(a) - p|B|)^2} + \sqrt{|A|(p|B| - \overline{d})^2}  \right).\\
\end{eqnarray*}

For $\delta > 0$ given we have that the probability that 
\[\frac{\sum_{a \in A} (\deg(a) - p|B|)^2}{|A|} \geq \delta^2 |B|^2\]
can be bounded by observing that $\deg(a)$ is binomially distributed with $|B|$ trials and success probability $p$. As the degree of each vertex in $A$ is independent and binomially distributed by Chernoff bound for each $a \in A$, 
\[\Pr(|\deg(a) - p|B|| \geq \delta |B|) \leq 2 \exp(-\frac{\delta^2|B|}{3p}).\] Thus with probability at least
\[1 - |A|2 \exp(-\frac{\delta^2|B|}{3p}),\]
\[\frac{\sum_{a \in A} (\deg(a) - p|B|)^2}{|A|} \leq \delta^2 |B|^2.\]

In our case with overwhelming probability $|A| = \Theta(n^{1 - d\alpha})$ and $|B| = \Theta(n^{1 - d\alpha})$, so the above inequality holds with overwhelming probability.

For the second term, $\sqrt{|A|(p|B| - \overline{d})^2}$ we observe that
\[(p|A||B| - |A|\overline{d})^2 = (p|A||B| - e(A, B))^2.\]
By what we previously showed this is at most $\delta|A|^2 p |B|/4$ with overwhelming probability by a routine application of Chernoff bound, basically by the same argument that we used earlier to bound $|e(U, V) - p|U||V||$. Therefore with overwhelming probability
\[\sigma_A \leq \delta|B| + \frac{\delta \sqrt{|B|}}{2}.\]
Thus for $n$ sufficient large, $\sigma_A/|B|$ can be made smaller than any positive constant. The same holds for $\sigma_B/|A|$ and so the proof is complete.
\end{proof}

Lemma \ref{SymmetricCase} together with Garland's method essentially implies the statement about $2d$-homology in Theorem~\ref{RationalHomologyVanishingTheorem}. To give the full proof of Theorem~\ref{RationalHomologyVanishingTheorem} we need an analogous statement for the asymmetric case for~$\Pi_{d, d + 1}$. Within $\Pi_{d, d + 1}$, $\lk(\tau*\sigma)$ with $|\tau| = d-1$ and $|\sigma| = d$ is $H \sim H(n - (2d - 1), n^{-(d-1)\alpha}(1 - n^{-\alpha})^d, n^{-d\alpha}(1 - n^{-\alpha})^{d -1}, 0, 0, 1 - n^{-\alpha})$. With this motivation our main lemma for the asymmetric case is:
\begin{lemma}\label{AsymmetricCase}
For $d \geq 1$ and $\alpha < 1/(d + 1)$ and $\delta > 0$ with overwhelming probability $H \sim H(n - (2d - 1), n^{-(d-1)\alpha}(1 - n^{-\alpha})^d, n^{-d\alpha}(1 - n^{-\alpha})^{d -1}, 0, 0, 1 - n^{-\alpha})$ has spectral gap of the normalized Laplacian larger than $1 - \delta$.
\end{lemma}
\begin{proof}
As in the proof of the symmetric case we can treat the sizes of $A$ and $B$ as basically fixed. More precisely we have that with overwhelming probability, 
\[1.01 n^{1 - (d - 1)\alpha} > |A| > 0.99 n^{1 - (d - 1)\alpha}\]
and
\[1.01 n^{1 - d\alpha} > |B| > 0.99 n^{1 - d\alpha}.\]

Under these overwhelmingly likely assumptions we apply Theorem \ref{bipartiteexpansion} when each edge between $A$ and $B$ is included with probability $1 - n^{-\alpha}$. Let $U \subseteq A$ and $V \subseteq B$  be fixed and we bound the probability that 
\[\left| e(U, V) - \frac{|U||V|\overline{c}}{|A|}\right| \leq \sqrt{C|A|p|U||V|} \]
for $C$ a fixed constant. The proof that this holds with probability at least $1 - \exp(-C_1 |A|)$ is exactly the same application of Chernoff bound as in the proof of Lemma \ref{SymmetricCase}. By taking $C$ sufficiently large, $C_1$ becomes large enough that we can take the union bound over all $2^{|A| + |B|}$ choices of $|U|$ and $|V|$. Thus just as in the symmetric case we have that with overwhelming probability the spectral gap of our random bipartite graph is at least
\[1 - \frac{\sqrt{C |A|} \log(2(2|A| + |B|)/\sqrt{C|A|})}{0.99|B|} - \frac{2\sigma_A\sigma_B}{0.99|A|B|}\]
As $\sqrt{C|A|} \leq \sqrt{C 1.01 n^{1 - (d - 1)\alpha}}$ and $|B| \geq 0.99 n^{1 - d\alpha}$ with 
$(1 - (d - 1)\alpha)/2 < 1 - d \alpha$ for $\alpha < 1/(d + 1)$, the middle term is $o(1)$. For the last term, we apply Chernoff bound exactly as in the proof of Lemma \ref{SymmetricCase}.
\end{proof}

We now have everything in place to prove Theorem \ref{RationalHomologyVanishingTheorem}. 
\begin{proof}[Proof of Theorem \ref{RationalHomologyVanishingTheorem}]
Suppose that $d \geq 1$ and $\alpha < 1/d$. By Lemma \ref{GluingStep} with overwhelming probability every $2d$-cycle in $\widetilde Z$ for $\widetilde Z \sim \widetilde Z(2n, n^{-\alpha})$ is rationally homologous to a cycle in $\Pi_{d + 1, d + 1}$. We will use Garland's method to verify that $\Pi_{d + 1, d + 1}$ has no rational homology in degree $2d$. Let $\tau*\sigma$ be a $(2d - 1)$-face in $\Pi_{d + 1, d + 1}$, where $|\tau| \leq |\sigma|$ without loss of generality. Thus we have either $|\tau| = d$ and $|\sigma| = d$ or $|\tau| = d - 1$ and $|\sigma| = d + 1$. 

In the first case $\lk(\tau*\sigma)$ within $\Pi_{d + 1, d + 1}$ is $H(n - 2d, n^{-d\alpha}(1 - n^{-\alpha})^d, n^{-d\alpha}(1 - n^{-\alpha})^d, 0, 0, 1 - n^{-\alpha})$. In this case by Lemma \ref{SymmetricCase} we have that with overwhelming probability $\lk(\tau*\sigma)$ has large spectral gap. As this holds with overwhelming probability and the number of choices for $\tau * \sigma$ is trivially at most $n^{2d}$, with overwhelming probability every such link has large spectral gap.

In the second case $\lk(\tau*\sigma)$ within $\Pi_{d + 1, d + 1}$ is $H(n - 2d, n^{-(d - 1)\alpha}(1 - n^{-\alpha})^{d + 1}, 0, n^{-\alpha}, 0, 0)$. This is just the two parameter random graph model $G(n - 2d; n^{-(d - 1)\alpha}(1 - n^{-\alpha})^{d + 1}, n^{-\alpha})$ with $p_0p_1 = \Theta(n^{-(d - 1)\alpha - \alpha}) = \Theta(n^{-\alpha d})$ and so by Lemma \ref{twoparametergeneral} with overwhelming probability every such link has large spectral gap. Thus by Garland's method $H_{2d}(\Pi_{d + 1, d + 1}; \Q) = 0$.

Now suppose that $\alpha < 1/(d + 1)$. By Lemma \ref{GluingStep} it is enough to show that $H_{2d - 1}(\Pi_{d, d + 1} \cup \Pi_{d+1, d}; \Q) = 0$. This will follow by induction if $H_{2d-1}(\Pi_{d, d + 1}; \Q) = 0$. Indeed $\Pi_{d + 1, d}$ is isomorphic as a simplicial complex to $\Pi_{d, d + 1}$ so if $H_{2d-1}(\Pi_{d, d + 1}; \Q)$ then so is $H_{2d-1}(\Pi_{d + 1, d}; \Q)$. The intersection of $\Pi_{d, d + 1}$ and $\Pi_{d + 1, d}$ with overwhelming probability is $\Pi_{d, d}$ (this invokes the fact that in the original flag complex every simplex of dimension at most $d$ belongs to a $d$-simplex when $\alpha < 1/d$). Thus we have the Mayer--Vietoris sequence
\[H_{2d - 1}(\Pi_{d + 1, d}; \Q) \oplus H_{2d - 1}(\Pi_{d, d + 1}; \Q) \rightarrow H_{2d - 1}(\Pi_{d + 1, d} \cup \Pi_{d, d + 1}; \Q) \rightarrow H_{2d - 2}(\Pi_{d, d}; \Q).\]
As $\alpha < 1/(d + 1) < 1/(d - 1)$, $H_{2d - 2}(\Pi_{d, d}) = 0$ with overwhelming probability. (If $d = 1$ we have to check that $\tilde{H}_0(\Pi_{1, 1})$, in this case $\Pi_{1, 1}$ is just a random bipartite graph on two copies of the original vertex set with an edge from $+u$ to $-v$ if $u$ is \emph{not} adjacent to $v$ in the original complex. Each vertex in the original complex has small degree when $\alpha > 0$ so it is routine to verify that this dense, random symmetric bipartite graph is connected with high probability.)

By induction therefore, with overwhelming probability $$H_{2d - 1}(\Pi_{d + 1, d} \cup \Pi_{d, d + 1}; \Q) = 0$$ if $H_{2d - 1}(\Pi_{d, d + 1}; \Q) = 0$. To verify that $H_{2d - 1}(\Pi_{d, d + 1}; \Q) = 0$ we observe that the link of a $2d - 2$ face $\tau *\sigma$ with $|\tau| = d - 1$ and $|\sigma| = d$ is $H(n - (2d - 1), n^{-(d - 1)\alpha}(1 - n^{-\alpha})^d, n^{-d\alpha}(1 - n^{-\alpha})^{d - 1}, 0, 0, 1 - n^{-\alpha}$ which is handled by Lemma \ref{AsymmetricCase}. If $|\tau| = d$ and $|\sigma| = d - 1$ then the link within $\Pi_{d, d + 1}$ is a two parameter random graph $G(n - (2d - 1); n^{-(d - 1)\alpha}(1 - n^{-\alpha})^d, n^{-\alpha})$ with $p_0p_1 = \omega(\log n/n)$ so Lemma \ref{twoparametergeneral} applies. If $|\tau| = d - 2$ then $|\sigma| = d + 1$ then the link within $\Pi_{d, d + 1}$ is $G(n - (2d - 1); n^{-(d-2)\alpha}(1 - n^{-\alpha})^{d + 1}, n^{-\alpha})$ with $n^{-d\alpha}n^{-\alpha} = n^{-(d - 1)\alpha} = \omega(\log n/n)$, so in this final case again Lemma \ref{twoparametergeneral} applies. Thus by Garland's method $H_{2d - 1}(\Pi_{d, d + 1}; \Q) = 0$.
\end{proof}
\begin{proof}[Proof of (2) and (3) of Theorem \ref{DZtheorem} and (3) of Theorem \ref{maintheorem}]
Points (2) and (3) of Theorem \ref{DZtheorem} follow immediately from Theorem \ref{RationalHomologyVanishingTheorem} and the Hurewicz Theorem (Theorem~\ref{thm:hurewicz}). 
Part (3) of Theorem \ref{maintheorem} follows from Theorem~\ref{thm:iso-hg}.
\end{proof}
\section{Homology at and above dimension $2d + 1$}\label{sec:CritialDim}
The goal of this section is to prove that $H_{2d + 1}(\widetilde Z; \Q) \neq 0$ for $\widetilde Z \sim \widetilde Z(2n, n^{-\alpha})$ and $1/(d + 1) < \alpha < 1/d$, but that integer homology in higher dimensions is vanishing, and the same is true for $Z \sim Z(n, n^{-\alpha})$. All of this is relatively straightforward. We start with the statement about homology above dimension $2d + 1$.

\begin{theorem}\label{CollapsiblityTheorem}
For $d \geq 1$ and $\alpha > 1/(d + 1)$ with high probability $H_k(Z) = H_{k}(\widetilde Z) = 0$ for $k \geq 2d + 2$, $Z \sim Z(n, n^{-\alpha})$, and $\widetilde Z \sim \widetilde Z(n, n^{-\alpha})$.
\end{theorem}

The proof of this theorem is based on the fact that $\alpha > 1/(d + 1)$ implies that $X \sim X(n, n^{-\alpha})$ is $(d + 1)$-collapsible \cite{Malen, NewmanOneSideCollapse}. The following deterministic statement holds and so if $X$ collapses to a $d$-dimensional complex then the separated deleted join of $X$ collapses to a $2d + 1$ complex.
\begin{lemma}\label{LiftCollapsing}
If $(f, f \cup \{v\})$ is an elementary collapse in $X$ then there exists a sequence of elementary collapses in $\mdeljoin{X}$ to $\mdeljoin{(X \setminus \{f, f \cup \{v\}\})}$.
\end{lemma}
\begin{proof}
We want to show that if we can collapse away $(f, f \cup \{v\})$ in $X$ then we can collapse away exactly all faces $f * \Sigma_f \cup \Sigma_f*f$. For any maximal face $\sigma$ which is maximal as the negative side of $f*\sigma$ in $f * \Sigma_f$, we have that $f*\sigma$ is a free face in $\mdeljoin{X}$ as it only  belongs to $(f \cup \{v\}) * \sigma$, so it may be collapsed away. We can do this iteratively to get rid of all nonempty faces in $\Sigma_f$ by elementary collapses and then finally collapse $f$ into $f \cup \{v\}$ on the plus side. We can then do the same thing for $\Sigma_f * f$. In this way we eliminate the subcomplex of all faces that contain $f*\emptyset$ or $\emptyset*f$. Thus we collapse $\mdeljoin{X}$ to $\mdeljoin{(X \setminus \{f, f \cup \{v\}\})}$.
\end{proof}

\begin{proof}[Proof of Theorem \ref{CollapsiblityTheorem}]
With high probability $X \sim X(n, n^{-\alpha})$ is $(d + 1)$-collapsible for $\alpha > 1/(d + 1)$ by the main result of \cite{Malen}. In the case that it is $(d + 1)$-collapsible, $\mdeljoin{X}$ collapses to a complex of dimension at most $2d + 1$. As the collapses happen symmetrically $Z$ also collapses to a complex of dimension at most $2d + 1$.
\end{proof}

For homology in dimension $2d + 1$ we use the simple inequality that holds for any simplicial complex: 
\[\beta_{2d + 1} \geq f_{2d + 1} - f_{2d + 2} - f_{2d}.\]
Nonvanishing of $\beta_{2d + 1}(\widetilde Z; \Q)$ is then immediate by counting cliques in the graph of~$\widetilde Z$.

\begin{theorem}\label{FvectorCounts}
For $d \geq 1$, $\widetilde Z \sim \widetilde Z(2n, n^{-\alpha})$ if $\alpha < 1/d$ then with overwhelming probability $f_i(\widetilde Z) \gg f_{i - 1}(\widetilde Z)$ for $i \leq 2d + 1$, and if $1/(d + 1) < \alpha$ then with overwhelming probability $f_i(\widetilde Z) \ll f_{i - 1}(\widetilde Z)$ for $i \geq 2d + 2$. The same holds for $Z  \sim Z(n, n^{-\alpha})$ as well.
\end{theorem}

\begin{proof}
We count the expected number of $i$-dimensional faces in $\widetilde Z$ to be
\[\sum_{k + \ell = i + 1} 2\binom{n}{k}\binom{n - k}{\ell}(n^{-\alpha})^{\binom{k}{2} + \binom{\ell}{2}}(1 - n^{-\alpha})^{k\ell} = \Theta\left(n^{i + 1} n^{-\alpha \left(\binom{\lfloor (i + 1)/2 \rfloor}{2} + \binom{\lceil (i + 1)/2 \rceil}{2} \right)}\right).\]
Now we verify that 
\[i + 1 - \alpha \left(\binom{\lfloor (i + 1)/2 \rfloor}{2} + \binom{\lceil (i + 1)/2 \rceil}{2} \right)\]
is increasing in $i$ with $\alpha < 1/d$ and $i \leq 2d + 1$ and decreasing in $i$ with $1/(d + 1) < \alpha$ and $i \geq 2d + 2$. We consider the difference:
\begin{eqnarray*}
i + 1 - \alpha \left(\binom{\lfloor (i + 1)/2 \rfloor}{2} + \binom{\lceil (i + 1)/2 \rceil}{2} \right) - \left(i - \alpha \left(\binom{\lfloor i/2 \rfloor}{2} + \binom{\lceil i/2 \rceil}{2} \right)\right).
\end{eqnarray*}
This is equal to 
\begin{eqnarray*}
1 - \alpha \left( \binom{\lceil (i + 1)/2 \rceil}{2} - \binom{\lfloor i/2 \rfloor}{2} \right) &=& 1 - \alpha \left( \binom{\lfloor i/2 \rfloor + 1}{2} - \binom{\lfloor i/2 \rfloor}{2} \right) \\
&=& 1 - \alpha \left(\lfloor i/2 \rfloor \right).
\end{eqnarray*}
This is positive if 
\[\alpha < \frac{1}{\lfloor i/2 \rfloor}\]
and nonpositive otherwise. Thus for $\alpha < 1/d$ we have that the number of $i$-dimensional faces is increasing in $i$ for $i \leq 2d + 1$. On the other hand if $\alpha > 1/(d + 1)$ then the number of $i$-dimensional faces is increasing in $i$ for $i \geq 2d + 2$. To this point we have just proved that the expected value of the $i$-dimensional faces follows the unimodal pattern that we want. However, standard concentration inequalities for the number of cliques in an Erd\H{o}s--R\'{e}nyi random graph show that the $f_i$'s grow at the same rate as their expected values with overwhelming probability (c.f. Theorem 5.6 of \cite{FriezeRandomGraphs}). 

Finally as $f_i(\widetilde Z) = 2f_i(Z)$ the exact same inequalities hold for~$Z$.
\end{proof}

\begin{proof}[Proof of points 4 and 5 of Theorems \ref{maintheorem} and \ref{DZtheorem}]
In both theorems point 4 follows by Theorem \ref{FvectorCounts} and point 5 by Theorem \ref{CollapsiblityTheorem}.
\end{proof}

\section{The missing homology group}
\label{MissingGroup}

We have shown that for $1/(d + 1) < \alpha < 1/d$, $H_i(\widetilde Z; \Q) =0$ for $i \leq 2d -2$, $i = 2d$, and $i \geq 2d + 2$. We also verified that $H_{2d + 1}(\widetilde Z; \Q) \neq 0$. However, we do not know what happens with $H_{2d - 1}(\widetilde Z)$. Here we discuss why attempting to apply Garland's method to show that this homology group vanishes with $1/(d + 1) < \alpha$ does not work the same way as the homology vanishing statement for $H_{2d}(\widetilde Z)$ or the homology vanishing statement for $H_{2d - 1}(\widetilde Z)$ when $\alpha < 1/(d + 1)$. 

Recall that the strategy for Theorem \ref{RationalHomologyVanishingTheorem} involved finding a ``with overwhelming probability" bound for vertex expansion of a random bipartite graph and then applying Theorem \ref{bipartiteexpansion} to use strong vertex expansion to derive a bound on the spectral expansion. Proving a homology vanishing statement for the $(2d - 1)$st homology group when $\alpha > 1/(d + 1)$ could be accomplished by proving that we actually \emph{do} have strong enough vertex expansion in the case of the resulting random bipartite graph for Theorem \ref{bipartiteexpansion} to apply. Alternatively, maybe a stronger version of Theorem \ref{bipartiteexpansion} holds which still gives the same conclusion on spectral expansion but with some weaker assumption on vertex expansion.

We do not address the latter case here, but for the former case we demonstrate that new arguments would be needed to establish sufficiently strong vertex expansion. We first contrast the ``symmetric case" that the $2d$th homology group is vanishing for $\alpha < 1/d$ in the statement of Theorem \ref{RationalHomologyVanishingTheorem} with the ``asymmetric case" about the $(2d - 1)$st homology group in the statement of Theorem \ref{RationalHomologyVanishingTheorem}. ``Symmetric" and ``asymmetric" here refer to the relative sizes of the two partite sets $A$ and $B$ in Lemmas \ref{SymmetricCase} and \ref{AsymmetricCase}. With the bounds we have from Theorem \ref{bipartiteexpansion}, we see that in both cases we compare $|A|$ with $|B|$ and in order to establish concentration on the spectral expansion we need to have $\sqrt{|A|} \ll |B|$. In the symmetric case, $A$ and $B$ both have size $\Theta(n^{1 -d\alpha})$. In particular they are both about the same size, hence the name "symmetric case" and clearly $\sqrt{|A|} \ll |B|$ when $\alpha < 1/d$. By contrast in the asymmetric case $|A| = \Theta(n^{1 - (d - 1)\alpha})$ and $|B| = \Theta(n^{1 - d \alpha})$ so to have $\sqrt{|A|} \ll |B|$ we have to have $\alpha < 1/(d + 1)$. The $\sqrt{|A|}$ here is coming from the vertex expansion bound, thus to improve the result we might want to see if we can get a better vertex expansion bound. This reduces to the following question about random bipartite graphs. 
\begin{question}
Fix $\epsilon \in [0, 1)$, suppose we have a random bipartite graph with $n$ vertices in $A$, $m$ vertices in $B$ with $m \approx n^{1 - \epsilon}$ and each edge between $A$ and $B$ included independently with probability $p$ where $p = 1 - n^{-\epsilon}$. Is it true that there is some $\delta > 0$ so that with overwhelming probability every $U \subseteq A$ and every $V \subseteq B$ satisfies
\[|e(U, V) - p|U||V|| \leq K|B|^{1 - \delta} \sqrt{|U||V|}\]
for some large absolute constant $K$?
\end{question}
We do not know the answer to this question, but we do know that there is a simple proof by a union bound that works for $\epsilon < 1/2$, but does not work for $\epsilon \geq 1/2$. This will ultimately imply that applying Garland's method to prove a homology vanishing statement about the $(2d-1)$st homology group is definitely \emph{not} simply a matter of finding a better tail bound for estimating vertex expansion. A proof using Garland's method that the $(2d -1)$st homology group vanishes for $1/(d + 1) < \alpha < 1/d$ requires different ideas from what we have done already.

We first show an affirmative answer to the question above for $\epsilon < 1/2$. By Chernoff bound we can find a $\delta$ that works for the upper tail for any $\epsilon$:
\begin{eqnarray*}
\Pr(e(U, V) \geq p|U|V| + K|B|^{1 - \delta}\sqrt{|U||V|}) &=& \Pr(|U||V| - e(U, V) \geq (1 - p)|U||V| - K|B|^{1 - \delta}\sqrt{|U||V|})  \\
&\leq& \exp(-K^2|B|^{2 - 2\delta}|U||V|/(2(1 - p)|U||V|)) \\ 
&\leq& \exp(-0.49K^2 |B|^{2 - 2\delta}/(1 - p)) \\
&=& \exp(-0.49 K^2 |B|^{2 - 2\delta + \epsilon/(1 - \epsilon)})
\end{eqnarray*}
for $n$ large enough.
Now the number of choices for $U$ and $V$ is at most $2^{|A| + |B|}$ so by a union bound the probability that some pair $U$ and $V$ has way more edges going across than expected is at most
\[\exp(2|A| \log 2 - 0.49K^2 |B|^{2 -2 \delta + \epsilon/(1 - \epsilon)}).\]
So we can take any $\delta \leq 1/2$ regardless of $\epsilon$ and the upper tail estimate holds for all choices of $U$ and $V$.

For the lower tail
\[\Pr(e(U, V) \leq p|U||V| - K|B|^{1 - \delta}\sqrt{|U||V|}) \leq \exp(-0.49 K^2|B|^{2 - 2\delta})\]
for $n$ large enough. This survives the union bound with all $2^{|A| + |B|}$ sets if we can pick $\delta$ so that $(1 - \epsilon)(2 - 2\delta) \geq 1$ (with equality requiring setting $K$ to be a large constant), so this runs into a problem if $\epsilon \geq 1/2$, but is fine if $\epsilon < 1/2$.

When $\epsilon \geq 1/2$, however we can show that we cannot take a simple union bound over all $\exp(\Theta(n))$ subsets $U$ and $V$, as the following claim shows.
\begin{claim}
For $\epsilon \geq 1/2$ if $|A| = n$, $|B| = n^{1 - \epsilon}$ and $p = 1 - n^{-\epsilon}$ there is $U \subseteq A$ and $V \subseteq B$ so that for any $\delta > 0$
\[\Pr(e(U, B) \leq p|U||V| - K|B|^{1 - \delta}\sqrt{|U||V|}) \geq \exp(-o(n))\]
\end{claim}
\begin{proof}
For any $|U|$ and $|V|$ the probability that $e(U, V) = 0$ is clearly $(n^{-\epsilon})^{|U||V|}$. Now
\begin{eqnarray*}
n^{-\epsilon |U||V|} \geq \exp(-o(n))
\end{eqnarray*}
whenever
\[|U||V| \leq o \left( \frac{n}{\log n} \right).\]
On the other hand $e(U, V) = 0$ is sufficient to conclude that 
$e(U, V) \leq p|U||V| - K|B|^{1 - \delta}\sqrt{|U||V|}$ whenever 
\[p|U||V| \geq K|B|^{1 - \delta}\sqrt{|U||V|}\]
i.e. for
\[|U||V| = \Theta(n^{2(1 - \epsilon)(1 - \delta)})\]
So if 
\[1 > 2(1 - \epsilon)(1 - \delta)\]
Then we can find a $|U||V|$ that satisfies both inequalities. As $(1 - \epsilon) \leq 1/2$ this holds for any $\delta > 0$. 
\end{proof}

From this claim we see that while it may turn out to be true that a random bipartite graph with $n$ vertices on one side and $n^L$ vertices on the other, for $L$ arbitrarily large, and $p$ close to 1 is a spectral expander {with overwhelming probability}, the proof when $L \geq 2$ requires a different approach than the approach we take with $L < 2$. 
We leave the following as open question:
\begin{itemize}
    \item Can the rational homology vanishing statement in Theorem \ref{DZtheorem} be improved to integer homology vanishing statements?
    \item For $1/(d + 1) < \alpha < 1/d$ is $H_{2d - 1}(\widetilde Z; \Q) = 0$ and if it is can this be proved by Garland's method?
\end{itemize}

\section*{Acknowledgments}
The authors thank Alan Frieze and Peleg Michaeli for helpful discussion regarding one of the tail bound estimates for Lemma \ref{SymmetricCase}.
\end{document}